\documentclass{article}

\usepackage[utf8]{inputenc}
\usepackage[T1]{fontenc}
\usepackage[margin=1in]{geometry}
\usepackage[round]{natbib}
\usepackage{amsmath,amssymb,amsthm}
\usepackage{graphicx}
\usepackage{subcaption}
\usepackage{tikz}
\usetikzlibrary{arrows.meta,calc,backgrounds}
\usepackage{mathtools}
\usepackage{xr-hyper}
\usepackage{hyperref}
\hypersetup{hidelinks}
\usepackage{url}
\usepackage{booktabs}
\usepackage{nicefrac}
\usepackage{microtype}
\usepackage{xcolor}
\definecolor{helperorange}{RGB}{255,229,204}
\definecolor{helperdarkorange}{RGB}{190,95,0}

\usepackage{enumitem}
\usepackage{longtable}
\usepackage{bm}
\usepackage{dsfont}
\newcommand{\argmin}{\mathop{\mathrm{arg\,min}}}
\newcommand{\argmax}{\mathop{\mathrm{arg\,max}}}
\newcommand{\dotp}[2]{\langle #1,#2 \rangle}
\theoremstyle{plain}
\newtheorem{theorem}{Theorem}[section]
\newtheorem{lemma}{Lemma}[section]
\newtheorem{proposition}{Proposition}[section]
\newtheorem{corollary}{Corollary}[section]
\newtheorem{definition}{Definition}[section]
\newtheorem{remark}{Remark}[section]
\newcommand{\mc}[1]{\mathcal{#1}}

\newcommand{\ones}{\mathds{1}}

\newcommand{\Vertex}{V}
\newcommand{\Edge}{E}

\newcommand{\length}{W}

\newcommand{\eqdef}{\coloneqq}

\DeclareMathOperator{\diverg}{div}
\newcommand{\RR}{\mathbb{R}}

\newcommand{\Ll}{\mathcal{L}}

\newcommand{\Cc}{\mathcal{C}}

\DeclareMathOperator{\Proj}{Proj}

\DeclareMathOperator{\diag}{diag}
\newcommand{\al}{\alpha}

\renewcommand{\epsilon}{\varepsilon}
\DeclareMathOperator{\sign}{sign}

\newcommand{\range}{\operatorname{range}}
\newcommand{\norm}[1]{\|#1\|}
\newcommand{\Umax}{U_\gamma}
\newcommand{\Xmax}{X_\gamma}
\newcommand{\XmaxZero}{X_0^\star}

\newcommand{\norminf}[1]{\norm{#1}_\infty}
\DeclareMathOperator{\arsinh}{arsinh}
\newcommand{\normVar}[1]{\norm{#1}_{\mathrm{Var}}}
\newcommand{\z}{z}
\newcommand{\zC}{\z^C}
\newcommand{\KL}{\mathrm{KL}}
\newcommand{\KLdiv}[2]{\mathrm{KL}(#1|#2)}
\newcommand{\Xx}{\mathcal{X}}
\title{Robust Sublinear Convergence Rates \\ for Iterative Bregman Projections}
\author{Gabriel Peyr\'e\\CNRS and ENS, Universit\'e PSL\\\url{gabriel.peyre@ens.fr}}
\date{\today}
\begin{document}
\maketitle
\begin{abstract}
Entropic regularization provides a simple way to approximate linear programs whose constraints split into two or more tractable blocks. The resulting objectives are amenable to cyclic Kullback--Leibler (KL) Bregman projections, with Sinkhorn-type algorithms for optimal transport, matrix scaling, and barycenters as canonical examples.
This paper gives a general blueprint for proving  $O(1/k)$ dual convergence rate with a constant that scales only \emph{linearly} in $1/\gamma$, where $\gamma$ is the entropic regularization parameter. We call such rates ``robust'', because this mild dependence on $\gamma$ underpins favorable complexity bounds for approximating the unregularized problem via alternating KL projections.
The blueprint reduces the proof to a uniform primal bound and a dual bound for a quotient norm induced by the constraint split. To make these inputs usable, we propose two helper results, which rely on the non-expansiveness of the dual iterations in this quotient dual norm. 
Instantiating this blueprint for graph-structured transport yields a new flow-Sinkhorn algorithm for the Wasserstein-1 distance on graphs. It achieves $\epsilon$-additive accuracy on the transshipment cost in $O(p\,\mathrm{diameter}^3/\epsilon^{4})$ arithmetic operations (up to logarithmic factors), where $p$ is the number of edges.
We also provide a machine-checked Lean formalization of the core blueprint and its graph-$\mathrm{W}_1$ instantiation.
\end{abstract}

\section{Introduction}

Iterative Bregman projections power many scalable ML solvers: they turn entropy-regularized constrained problems into alternating normalization steps implemented with sparse linear algebra and GPU-friendly tensor operations. They underlie Sinkhorn-type methods for optimal transport~\citep{Cuturi13} and their generalizations, such as barycenters~\citep{benamou2015iterative} and unbalanced-OT~\citep{chizat2018scaling}. This paper provides a blueprint for proving sublinear convergence with constants that remain ``robust'' at small regularization. The Wasserstein-1 distance on graphs is one concrete instance of this blueprint: the graph structure changes the split and the dual quotient norm with respect to the vanilla Sinkhorn analysis, but the proof still follows the same route.

\paragraph{Entropic regularization of structured linear programs.}
We consider a feasible linear program $\min_{Ax=b,\,x\ge0}\dotp{x}{C}$ whose constraints $Ax=b$ naturally split into blocks.
We focus on the two-block case for clarity.
Throughout, let $d \in\mathbb{N}$ be the dimension of the problem and write $\RR^d_{+}$ for the positive orthant over which the optimization is carried out. Let $A=(A_1;A_2)$, with $A_1\in\RR^{m_1\times d}$, $A_2\in\RR^{m_2\times d}$ and $m=m_1+m_2$, and split $b\in\RR^m$ as $b=(b_1;b_2)$ with $b_1\in\RR^{m_1}$ and $b_2\in\RR^{m_2}$.
We define the two affine constraint sets
\[
  \Cc_{1}\eqdef\{x\in\RR^{d}_{++}: A_{1}x=b_{1}\},
  \qquad
  \Cc_{2}\eqdef\{x\in\RR^{d}_{++}: A_{2}x=b_{2}\}.
\]
We assume $\Cc_{1}\cap\Cc_{2}\neq\varnothing$, and in particular $b \in \range(A)$.
This linear program can thus be re-written as
\begin{equation}
	\label{eq:linprog}
	\tag{$\mc P_{0}$}
  \min_{x\in\RR^{d}_{+}} \{ \dotp{C}{x}
  \quad\text{s.t.}\quad
  A_{1}x=b_{1},\ \ A_{2}x=b_{2} \}
  \;=\;
  \min_{x\in\Cc_{1}\cap\Cc_{2}} \ \dotp{C}{x}
\end{equation}
for some cost $C \in \RR^d$.
Fix a reference vector $\z\in\RR^{d}_{++}$.
For a temperature $\gamma>0$, we consider the following entropically regularized problem, which aims to approximate the original feasible linear program using fast iterative schemes
\begin{equation}
	\label{eq:entropic-penalized}
	\tag{$\mc P_{\gamma}$}
	\min_{x\in\Cc_{1}\cap\Cc_{2}}
	      \dotp{C}{x}+\gamma\,\KLdiv{x}{\z}, \quad\text{where}\quad
	      \KLdiv{x}{\z} \eqdef \sum_{i=1}^{d}\Bigl(x_{i}\log\frac{x_{i}}{\z_{i}}-x_{i}+\z_{i}\Bigr).
\end{equation}
It is common to use a constant reference $z=\alpha \ones_d$ for some scale parameter $\alpha \in \RR_+$. Since we do not restrict our attention to probability vectors $x$ (as in optimal transport), the choice of $\alpha$ is meaningful and should reflect prior knowledge about the expected total mass of the solution.
The regularised objective admits a simple \emph{cyclic Bregman projection} scheme detailed in Section~\ref{sec:general-KL-new}: project $x$ alternately onto $\Cc_{1}$ and $\Cc_{2}$ in the Kullback--Leibler divergence.
This method is attractive when both projections are explicit; the classical Sinkhorn algorithm for optimal transport is the canonical example.

This paper establishes sharp sub-linear $O(1/(\gamma k))$ convergence bounds for the \emph{dual objective} of the KL-projection scheme and demonstrates their impact through a new algorithm for the Wasserstein--1 distance on graphs.

\paragraph{Iterative Bregman projection.}
The idea of alternating Bregman projections originates from the work of Bregman~\citep{Bregman67} and of Csiszár--Tusnády~\citep{CsiszarTusnady84}. When the distance is squared Euclidean, this procedure reduces to von Neumann's alternating projections, whose linear convergence under an angle condition was already established by Friedrichs~\citep{Friedrichs38}. More generally, asymptotic linear rates for Bregman projections were shown in~\citep{CensorRezac15}, again assuming a qualification property that, in Euclidean settings, corresponds to such an angle condition. This assumption fails for the KL divergence unless all coordinates of the iterates remain bounded away from zero, something that, in many applications such as entropically regularised optimal transport, would require a bound of order $1-e^{-\norm{C}_\infty/\gamma}$, which becomes prohibitively small when $\gamma$ is small. These earlier convergence results focus on bounding the distance between primal iterates and the solution, typically in $\ell^1$ norm for the KL divergence. Our interest instead lies in controlling the \emph{dual} value, which provides a weaker guarantee but leads to sharper constants, avoiding the exponential blow-up in $1/\gamma$ and making the analysis relevant in high-dimensional machine learning contexts.

\paragraph{Entropic regularisation and Sinkhorn-type methods.}
The link between entropy and transport can be traced back to Schrödinger's 1931 formulation of the Brownian bridge problem. In the modern computational setting, Cuturi~\citep{Cuturi13} popularised entropic regularisation for large-scale optimal transport, showing that the resulting smoothed problems bypass the curse of dimensionality in empirical settings. The associated iterative scaling scheme was introduced independently multiple times: it appears as early as Yule's work in 1912~\citep{yule1912methods} and was later introduced and analyzed by Sinkhorn~\citep{Sinkhorn64}, as well as by Deming--Stephan~\citep{DemingStephanIPFP}. As reviewed in~\citep{chizat2025sharper}, convergence analyses for these iterations generally fall into two categories. 
The first comprises \emph{linear rates} (i.e., exponential with the iteration index $k$) results with constants that degrade exponentially badly when $\gamma$ is small, which we call non-robust rates following~\citep{chizat2025sharper}. One approach uses Hilbert's projective metric to show contraction~\citep{FranklinLorenz89,borwein1994dad}, with rates of the order $[ 1 - e^{-\norm{C}_\infty/\gamma} ]^k$; see~\citep{chen2016hilbertmetric,deligiannidis2024hilbertmetric} for continuous-domain extensions and~\citep{eckstein2024hilbertsprojectivemetricfunctions} for non-compact domains. Equivalent dependencies arise from convex optimisation proofs~\citep{marino2020schrodinger,carlier2022multimarginalsinkhorn}, which extend to multi-marginal settings~\citep{greco2023coupling,conforti2023quantitative}. For arbitrary, possibly adversarial, cost, this dependence on $\norm{C}_\infty/\gamma$ appears tight. 
The second category covers polynomial-rate bounds (typically $1/k$) with constants that remain stable as $\gamma$ changes (what we call ``robust'' rates). Early complexity results of this type go back to~\citep{kalantari2008complexity} and were sharpened in later works~\citep{altschuler2017near,chakrabarty2021sinkhornsublinearrate,dvurechensky2018computational,chizat2020faster}. One obtains rate on the dual objective of the order $\norm{C}_\infty^2 / (\gamma k)$. This approach can be reframed as mirror descent in a tailored geometry~\citep{leger2021gradient,aubin2022mirror}.
In continuous cases (or in discrete cases if one accepts dependency on the number of Dirac masses), it is, however, possible to obtain the best of both worlds, and~\citep{chizat2025sharper} shows that when the marginals are bounded from below that the linear rate is of the order $[1 - \kappa \gamma^2/C_{\max}^2 ]^k$ for some constant $\kappa$. Our contributions focus on sublinear rates, and aim at understanding the general structure that enables robust rates, to apply this analysis to a larger class of linear programs.

\paragraph{Wasserstein-1 on graphs.}  
The Wasserstein-1 distance has a particular structure that makes it both more robust and more tractable than $W_2$: it is less sensitive to outliers, and it admits a flow formulation that is often faster to compute. It has been applied in computer vision~\citep{rubner1998metric,grauman2004fast}, machine learning~\citep{kusner2015word}, graphics~\citep{SolomonEMDSurfaces2014}, community detection~\citep{sia2019ollivier}, and biology~\citep{sandhu2015graph}. When the ground cost is the shortest-path distance on a graph, $W_{1}$ can be formulated as a minimum-cost flow problem with a fixed divergence constraint~\citep{Beckmann52}, see for instance~\citep{carsan10} for applications of this framework to PDEs on continuous domains. Classical solvers such as the network simplex can be used, and Orlin's strongly polynomial algorithm achieves $O(p n\log n)$ time on a graph with $n$ vertices and $p$ edges~\citep{AhujaEtAl93}. The current best exact min-cost flow algorithm runs in $O(p^{1+o(1)})$ time with high probability~\citep{chen2025maximum}. For planar graphs with polynomially bounded integer data, a nearly linear-time exact algorithm is known~\citep{dong2025nested}, and for tree metrics (generalizing the 1-D case), $W_1$ can be computed in linear time $O(n)$~\citep{evans2012phylogenetic}. For approximate solvers, which are the focus of this paper, interior-point methods combined with fast Laplacian solvers yield additive-$\varepsilon$ approximations for min-cost generalized flows in $O(p^{3/2}\log(1/\varepsilon))$ time~\citep{daitch2008faster}. In this work, we propose a simpler alternative based on entropic regularization, with complexity $O(p\,\mathrm{diameter}(\Edge)^3/\varepsilon^{4})$ (up to logarithmic factors). While this has a worse dependence on $\varepsilon$, it is easy to implement (including on GPUs) and scales linearly with the number of edges $p$.

\begin{center}
\refstepcounter{figure}\label{fig:blueprint}
\resizebox{\linewidth}{!}{\begin{tikzpicture}[x=1mm,y=-1mm,
  font=\scriptsize,
  boxblue/.style={draw=none,rounded corners=1.1mm,fill=blue!28,inner sep=.75mm,align=center,text=blue!45!black},
  boxgray/.style={draw=black!70,rounded corners=1.9mm,fill=gray!18,line width=.55pt,inner sep=.75mm,align=center,text=black},
  boxred/.style={draw=none,rounded corners=1.8mm,fill=red!23,inner sep=.8mm,align=center,text=black},
  boxorange/.style={draw=none,rounded corners=1.8mm,fill=orange!20,inner sep=.8mm,align=center,text=black},
  bluearr/.style={-{Stealth[length=2.30mm,width=1.80mm]},blue!58,line width=.82pt},
  redarr/.style={-{Stealth[length=2.30mm,width=1.80mm]},red!55,line width=.82pt},
  header/.style={font=\bfseries\normalsize,text=green!42!black,align=center},
  helper/.style={font=\bfseries\normalsize,align=center}
]
\path[use as bounding box] (0,0) rectangle (160,90);

\begin{scope}[on background layer]
  \fill[green!17,rounded corners=9mm] (132.8,.0) rectangle (159.7,18.7);
  \fill[yellow!25,rounded corners=9mm] (73.8,.0) rectangle (132.4,21.4);
  \fill[cyan!17,rounded corners=8mm] (40.6,21.6) rectangle (85.0,62.3);
  \fill[green!17,rounded corners=9mm] (0,.0) rectangle (40.0,50.4);
  \fill[red!15,rounded corners=12mm] (97.3,53.0) rectangle (160.0,90.0);
  \fill[orange!18,rounded corners=10mm] (0.0,64.5) rectangle (97.3,90.0);
\end{scope}

\node[header,anchor=north east,text width=26.0mm,font=\bfseries\small] at (158.4,1.2) {Specific for KL};
\node[helper,text=yellow!55!black,anchor=north,text width=42mm] at (103.2,1.1) {Primal bound helper};
\node[helper,text=teal!70!black,anchor=north,text width=36mm] at (63.0,23.1) {Dual bound helper};
\node[header,anchor=north west,text width=14.8mm] at (2.6,1.5) {Specific\\for $W_1$};
\node[font=\bfseries\normalsize,text=red!72!black] at (130.5,86.5) {Main results};
\node[font=\bfseries\normalsize,text=orange!70!black] at (51.0,86.6) {Non-expansiveness helpers};

\node[boxgray,text width=18.9mm,minimum width=20.0mm,minimum height=9.2mm] (pinsker) at (147.0,11.2) {Lemma~\ref{app-lem:pinsker-nonnormalized}\\$\eta_\gamma=1/(2X_\gamma)$};
\node[boxblue,text width=18.8mm,minimum width=20.0mm,minimum height=9.0mm] (dphi) at (146.5,44.1) {$D_\phi<=\eta_\gamma\lvert\cdot\rvert^2$};
\node[boxgray,text width=28.2mm,minimum width=29.6mm,minimum height=12.0mm] (prop41) at (114.8,12.8) {Prop.~\ref{prop:mass-bound-block}\\[-.4mm]$X_\gamma=\dfrac{\lVert b\rVert_1U_\gamma}{\gamma}+de^{-C_{\min}/\gamma}$};
\node[boxblue,text width=24.0mm,minimum width=24.4mm,minimum height=7.1mm] (xbound) at (131.6,32.2) {$\lVert x^{(k)}\rVert_1\leq X_\gamma$};
\node[boxblue,text width=15.8mm,minimum width=16.8mm,minimum height=7.1mm] (ubound) at (109.4,32.2) {$\lVert u^{(k)}\rVert_\infty\leq U_\gamma$};
\node[boxblue,text width=14.6mm,minimum width=15.6mm,minimum height=6.4mm] (cminineq) at (84.8,13.0) {$C_i\geq C_{\min}$};
\node[boxblue,text width=16.2mm,minimum width=17.6mm,minimum height=7.2mm] (abound) at (95.9,45.7) {$\lvert A\rvert_{1\to1}$ bound};

\node[boxgray,text width=12.2mm,minimum width=13.5mm,minimum height=6.4mm] (thm32) at (139.6,60.8) {Thm~\ref{thm:approx-linprog}};
\node[boxred,text width=22.5mm,minimum width=23.5mm,minimum height=11.6mm] (eps) at (146.5,73.3) {$\Delta_k=O(\varepsilon)$\\[.7mm]$k=O(M_\gamma X_0/\varepsilon^2)$};
\node[boxgray,text width=12.5mm,minimum width=13.4mm,minimum height=6.4mm] (thm31) at (112.7,60.8) {Thm~\ref{thm:kl-dual-rate}};
\node[boxred,text width=28.0mm,minimum width=29.5mm,minimum height=11.8mm] (deltak) at (117.0,73.3) {$\Delta_k=O(M_\gamma/\gamma k)$\\[.7mm]$M_\gamma:=2U_\gamma^2\lvert A\rvert_{1\to1}^{2}/\eta_\gamma$};

\node[boxblue,text width=25.0mm,minimum width=26.0mm,minimum height=6.4mm] (logbound) at (61.5,31.7) {$\lvert\log(x_\gamma/z)\rvert_\infty\leq H_\gamma$};
\node[boxgray,text width=25.2mm,minimum width=26.5mm,minimum height=9.6mm] (prop42) at (61.3,44.0) {Prop.~\ref{prop:uniform-iter-final}\\$U_\gamma=2\kappa(\lvert C\rvert_\infty+\gamma H_\gamma)$};
\node[boxblue,text width=20.0mm,minimum width=21.0mm,minimum height=8.9mm] (psiNE) at (71.5,58.0) {$\Psi$\\non-expansive};
\node[boxblue,text width=15.6mm,minimum width=16.5mm,minimum height=8.9mm] (kappa) at (50.8,58.0) {$\kappa(A_1,A_2)$\\def.~\eqref{eq:kappa}};

\node[boxgray,text width=13.4mm,minimum width=14.5mm,minimum height=6.4mm] (propG1) at (88.4,74.7) {Prop.~\ref{app-prop:topical-nonexpansive}};
\node[boxblue,text width=24.5mm,minimum width=25.4mm,minimum height=6.5mm] (psimono) at (62.4,69.4) {$\Psi$ monotone};
\node[boxblue,text width=25.0mm,minimum width=25.5mm,minimum height=6.6mm] (psitrans) at (62.4,80.2) {$\Psi(x+\mathbf{1})=\Psi(x)+\mathbf{1}$};
\node[boxgray,text width=13.6mm,minimum width=14.6mm,minimum height=6.5mm] (propG2) at (36.3,69.2) {Prop.~\ref{app-prop:block-monotone}};
\node[boxgray,text width=13.6mm,minimum width=14.6mm,minimum height=6.5mm] (propG3) at (36.3,80.2) {Prop.~\ref{app-prop:translation-equivariance}};
\node[boxblue,text width=16.0mm,minimum width=17.0mm,minimum height=6.7mm] (sigma) at (14.5,69.4) {$\Sigma$-signature};
\node[boxblue,text width=16.0mm,minimum width=17.0mm,minimum height=6.7mm] (taub) at (14.5,80.2) {$\tau$-balance};

\node[boxgray,text width=15.2mm,minimum width=16.5mm,minimum height=7.8mm] (wmin) at (27.2,5.5) {$C_{\min}=W_{\min}$};
\node[boxgray,text width=15.4mm,minimum width=16.6mm,minimum height=7.8mm] (aone) at (27.2,14.7) {$\lvert A\rvert_{1\to1}=1$};
\node[boxgray,text width=30.9mm,minimum width=32.5mm,minimum height=9.9mm] (propF5) at (19.7,25.0) {Prop.~\ref{app-prop:hgamma-graphw1}\\$H_\gamma=\log(\cdots/W_{\min})+2W_{\max}/\gamma$};
\node[boxgray,text width=25.2mm,minimum width=26.5mm,minimum height=8.5mm] (propF6) at (19.7,36.2) {Prop.~\ref{app-prop:kappa-graph-diameter}\\$\kappa=2\operatorname{diam}(E)$};
\node[boxgray,text width=16.0mm,minimum width=17.0mm,minimum height=7.0mm] (sigmaid) at (27.2,45.0) {$\Sigma=(\mathrm{Id},-\mathrm{Id})$};
\node[boxgray,text width=10.8mm,minimum width=12.2mm,minimum height=7.0mm] (tauplus) at (9.0,45.0) {$\tau=+1$};

\coordinate (cminE) at ($(cminineq.north east)!0.50!(cminineq.south east)$);
\coordinate (prop41W) at ($(prop41.north west)!0.50!(prop41.south west)$);
\coordinate (prop41Sleft) at ($(prop41.south west)!0.34!(prop41.south east)$);
\coordinate (prop41Sright) at ($(prop41.south west)!0.66!(prop41.south east)$);
\coordinate (xboundN) at ($(xbound.north west)!0.50!(xbound.north east)$);
\coordinate (xboundS) at ($(xbound.south west)!0.58!(xbound.south east)$);
\coordinate (uboundN) at ($(ubound.north west)!0.50!(ubound.north east)$);
\coordinate (uboundS) at ($(ubound.south west)!0.50!(ubound.south east)$);
\coordinate (uboundW) at ($(ubound.north west)!0.55!(ubound.south west)$);
\coordinate (dphiS) at ($(dphi.south west)!0.45!(dphi.south east)$);
\coordinate (thm31E) at ($(thm31.north east)!0.42!(thm31.south east)$);
\coordinate (thm31Nleft) at ($(thm31.north west)!0.25!(thm31.north east)$);
\coordinate (thm31Nmid) at ($(thm31.north west)!0.50!(thm31.north east)$);
\coordinate (thm31Nright) at ($(thm31.north west)!0.75!(thm31.north east)$);
\coordinate (aboundN) at ($(abound.north west)!0.50!(abound.north east)$);
\coordinate (aboundS) at ($(abound.south west)!0.55!(abound.south east)$);
\coordinate (thm32W) at ($(thm32.north west)!0.55!(thm32.south west)$);
\coordinate (deltakNE) at ($(deltak.north west)!0.78!(deltak.north east)$);

\coordinate (logS) at ($(logbound.south west)!0.50!(logbound.south east)$);
\coordinate (logW) at ($(logbound.north west)!0.52!(logbound.south west)$);
\coordinate (prop42N) at ($(prop42.north west)!0.50!(prop42.north east)$);
\coordinate (prop42E) at ($(prop42.north east)!0.52!(prop42.south east)$);
\coordinate (prop42Sleft) at ($(prop42.south west)!0.36!(prop42.south east)$);
\coordinate (prop42Sright) at ($(prop42.south west)!0.70!(prop42.south east)$);
\coordinate (psiNEN) at ($(psiNE.north west)!0.55!(psiNE.north east)$);
\coordinate (psiNES) at ($(psiNE.south west)!0.55!(psiNE.south east)$);
\coordinate (kappaN) at ($(kappa.north west)!0.50!(kappa.north east)$);
\coordinate (kappaW) at ($(kappa.north west)!0.50!(kappa.south west)$);

\coordinate (sigmaE) at ($(sigma.north east)!0.50!(sigma.south east)$);
\coordinate (taubE) at ($(taub.north east)!0.50!(taub.south east)$);
\coordinate (sigmaN) at ($(sigma.north west)!0.45!(sigma.north east)$);
\coordinate (taubN) at ($(taub.north west)!0.50!(taub.north east)$);
\coordinate (propG2W) at ($(propG2.north west)!0.50!(propG2.south west)$);
\coordinate (propG3W) at ($(propG3.north west)!0.50!(propG3.south west)$);
\coordinate (propG2E) at ($(propG2.north east)!0.50!(propG2.south east)$);
\coordinate (propG3E) at ($(propG3.north east)!0.50!(propG3.south east)$);
\coordinate (psimonoW) at ($(psimono.north west)!0.50!(psimono.south west)$);
\coordinate (psitransW) at ($(psitrans.north west)!0.50!(psitrans.south west)$);
\coordinate (psimonoE) at ($(psimono.north east)!0.50!(psimono.south east)$);
\coordinate (psitransE) at ($(psitrans.north east)!0.50!(psitrans.south east)$);
\coordinate (propG1Wupper) at ($(propG1.north west)!0.35!(propG1.south west)$);
\coordinate (propG1Wlower) at ($(propG1.north west)!0.67!(propG1.south west)$);
\coordinate (propG1N) at ($(propG1.north west)!0.45!(propG1.north east)$);

\coordinate (wminE) at ($(wmin.north east)!0.60!(wmin.south east)$);
\coordinate (wminW) at ($(wmin.north west)!0.60!(wmin.south west)$);
\coordinate (aoneW) at ($(aone.north west)!0.50!(aone.south west)$);
\coordinate (aoneE) at ($(aone.north east)!0.50!(aone.south east)$);
\coordinate (propF5E) at ($(propF5.north east)!0.52!(propF5.south east)$);
\coordinate (propF5Nleft) at ($(propF5.north west)!0.28!(propF5.north east)$);
\coordinate (propF6E) at ($(propF6.north east)!0.50!(propF6.south east)$);
\coordinate (propF6S) at ($(propF6.south west)!0.78!(propF6.south east)$);
\coordinate (sigmaidS) at ($(sigmaid.south west)!0.45!(sigmaid.south east)$);
\coordinate (tauS) at ($(tauplus.south west)!0.50!(tauplus.south east)$);

\draw[bluearr] (cminE) -- (prop41W);
\draw[bluearr] (uboundN) -- (prop41Sleft);
\draw[bluearr] (xboundS) -- (thm31Nright);
\draw[bluearr] (uboundS) -- (thm31Nmid);
\draw[bluearr] (aboundS) -- (thm31Nleft);
\draw[bluearr] (dphiS) -- (thm31E);
\draw[bluearr] (deltakNE) -- (thm32W);
\draw[bluearr] (logS) -- (prop42N);
\draw[bluearr] (psiNEN) -- (prop42Sright);
\draw[bluearr] (kappaN) -- (prop42Sleft);
\draw[bluearr] (sigmaE) -- (propG2W);
\draw[bluearr] (taubE) -- (propG3W);
\draw[bluearr] (psimonoE) -- (propG1Wupper);
\draw[bluearr] (psitransE) -- (propG1Wlower);
\draw[bluearr] (wmin.south west) to[out=230,in=102,looseness=.92] (propF5Nleft);

\draw[redarr] (pinsker.south) -- (dphi.north);
\draw[redarr] (prop41Sright) -- (xboundN);
\draw[redarr] (thm32.south) -- (eps.north);
\draw[redarr] (thm31.south) -- (deltak.north);
\draw[redarr] (wminE) -- (cminineq.west);
\draw[redarr] (aoneE) to[out=0,in=105] (abound.north);
\draw[redarr] (propF5E) -- (logW);
\draw[redarr] (prop42E) -- (uboundW);
\draw[redarr] (propF6E) -- (kappaW);
\draw[redarr] (sigmaidS) -- (sigmaN);
\draw[redarr] (tauS) -- (taubN);
\draw[redarr] (propG1N) -- (psiNES);
\draw[redarr] (propG2E) -- (psimonoW);
\draw[redarr] (propG3E) -- (psitransW);

\end{tikzpicture}
}
\smallskip
\parbox{\linewidth}{\small Figure~\thefigure: Logical map of the convergence blueprint and of its KL/flow-Sinkhorn instantiations. Blue arrows indicate hypotheses used by a theorem, while red arrows indicate the result statement of that theorem.}
\end{center}

\paragraph{Contributions.}
The core contribution of this paper is a general blueprint for analyzing the convergence speed of iterative Bregman projection methods summarized in Figure~\ref{fig:blueprint}, which presents the architecture of the paper and how the abstract blueprint is instantiated for the graph-$W_1$ flow-Sinkhorn algorithm. The main convergence results, shown in red, are Theorem~\ref{thm:kl-dual-rate} and Theorem~\ref{thm:approx-linprog}. They rely on primal and dual bounds handled by the helper results in yellow and cyan, including Proposition~\ref{prop:mass-bound-block} and Proposition~\ref{prop:uniform-iter-final}, and crucially on non-expansiveness of the dual update map $\Psi$, treated by the light-orange helper results in Appendix~\ref{app-sec:non-expansiveness}. The graph-$W_1$ instantiation, shown in green, yields the second core contribution detailed in  Section~\ref{sec:applications}: the flow-Sinkhorn algorithm, whose sparse updates scale linearly with the number of edges and lead to an explicit additive-accuracy complexity guarantee, together with numerical diagnostics on synthetic and genomic-inspired sparse graphs, including GPU line-scaling experiments, that illustrate the regularization/runtime tradeoff in the unregularized-accuracy regime. The implementation, benchmark scripts, and Lean formalization are available at \url{https://github.com/gpeyre/flow-sinkhorn}.
Appendices~\ref{app-sec:non-expansiveness} and~\ref{app-app:general-bregman} record auxiliary non-expansiveness results and general Bregman extensions.


\section{Iterative KL Projections}
\label{sec:general-KL-new}

We begin by casting any linear program with two affine constraints into an \emph{entropically regularised} form.  
The resulting objective is minimised by the classical \emph{cyclic KL--projection} (a.k.a.\ iterative Bregman projection) algorithm, whose dual convergence is analysed in Section~\ref{subsec:mainresult}---culminating in the $O(1/k)$ rate of Theorem~\ref{thm:kl-dual-rate}.  
Throughout the section, we keep the presentation self-contained and focus on two blocks for clarity. 

\paragraph{Entropic regularisation and cyclic KL projections.}
\label{subsec:entropic-regul-linprog}

Problems~\eqref{eq:entropic-penalized} can be conveniently re-written as a Bregman projection problem of the tilted reference $\zC$ (also called Gibbs kernel for OT problems)
\begin{equation}
	\label{eq:KL-primal}
	\min_{x\in\Cc_{1}\cap\Cc_{2}}\KLdiv{x}{\zC},
	\qquad
	\zC_{i}\eqdef \z_{i}e^{-C_{i}/\gamma}.
\end{equation}
The vector $\zC$ is the Gibbs kernel associated with $C$, $\gamma$ and $\z$.
Formulation \eqref{eq:KL-primal} reveals the structure exploited in the sequel: alternating KL projections onto $\Cc_{1}$ and $\Cc_{2}$ provides a scalable solver whose convergence is quantified in Section~\ref{subsec:mainresult}.

\label{subsec:cyclic}

Since the constraints are affine, problem~\eqref{eq:KL-primal} can be solved by iterative Bregman projection using the KL Bregman divergence.
Starting from $x^{(0)} = \zC$, it performs the two-step sweep
\begin{equation}
  x^{(k+\tfrac12)} \eqdef \argmin_{x\in\Cc_{1}}\KL(x|x^{(k)})\quad(P_1),
  \qquad
  x^{(k+1)} \eqdef \argmin_{x\in\Cc_{2}}\KL(x|x^{(k+\tfrac12)})\quad(P_2).
\end{equation}
When each projection has a closed form, or can be computed with only a few inexpensive inner iterations, $(P_1)$--$(P_2)$ provide an efficient outer loop. For instance, Appendix~\ref{app-app:sinkhorn-ot} revisits the classical optimal-transport splitting that yields Sinkhorn’s method, while Section~\ref{subsec-splitting-constr} introduces a new splitting for the Wasserstein-1 flow formulation on graphs. 

\paragraph{Dual problem.}
\label{subsec:primal-dual}
We now express the entropic programme \eqref{eq:entropic-penalized} in the dual variables by introducing the dual functional $F_\gamma$, which is the quantity used throughout the paper to measure convergence speed.

\begin{proposition}[Dual of \eqref{eq:entropic-penalized}]
\label{prop:dual-gamma-correct}
Let $\z\in\RR^d_{++}$ be fixed and set $Z \coloneqq \sum_{i=1}^d \z_i$. 
Then the primal and dual values satisfy
\begin{equation}
\label{eq:dual-regul}\tag{$\mathcal{D}_\gamma$}
  \min(P_\gamma)=\max_{u\in\RR^{m}} F_{\gamma}(u)
  \coloneqq
  \langle b,u\rangle+
  \gamma Z-
  \gamma \sum_{i=1}^d
  \z_i\exp\!\Bigl(\tfrac{(A^\top u)_i-C_i}{\gamma}\Bigr),
\end{equation}
where $\min(P_\gamma)$ denotes the value of~\eqref{eq:entropic-penalized}.
Moreover, any maximiser \(u^\star\) and the unique primal minimiser \(x^\star\) are linked by $x^\star=x(u^\star)$, where $x(u)_i\coloneqq \z_i\exp(((A^\top u)_i-C_i)/\gamma)=\zC_i\exp((A^\top u)_i/\gamma)$, $i=1,\dots,d$, and \(u^\star\) is characterised by the stationarity condition $\nabla F_\gamma(u^\star)=0\Longleftrightarrow A\,x(u^\star)=b$.
\end{proposition}

Proposition~\ref{prop:dual-gamma-correct} allows us to translate the cyclic projections $(P_1)$--$(P_2)$ into a block-coordinate ascent on $F_{\gamma}$ as we detail next. 
Note that while dual maximizers might not be unique, they are unique up to $\ker(A^\top)$, which is important to take into account in the analysis of the algorithm (for classical OT, this corresponds to translation of the dual potential).

\paragraph{Block--coordinate ascent in the dual.}

With the split $u=(u_{1},u_{2})\in\RR^{m_{1}}\times\RR^{m_{2}}$ we write $F_{\gamma}(u)=F_{\gamma}(u_{1},u_{2})$.  
The two primal KL projections $(P_1)$--$(P_2)$ are exactly one sweep of block maximisation on~$F_{\gamma}$: $u^{(k+1)}_1=\Psi_1(u^{(k)}_2)$ and $u^{(k+1)}_2=\Psi_2(u^{(k+1)}_1)$, where $\Psi_1(u_2)\eqdef\argmax_{u_1\in\RR^{m_1}}F_\gamma(u_1,u_2)$ and $\Psi_2(u_1)\eqdef\argmax_{u_2\in\RR^{m_2}}F_\gamma(u_1,u_2)$,
and the corresponding primal variable is obtained from Proposition~\ref{prop:dual-gamma-correct}:
$x^{(k)}=x(u_{1}^{(k)},u_{2}^{(k)})$. We define the full sweep map on the $u_1$ variable as $\Psi \eqdef \Psi_1 \circ \Psi_2$ so that $u_1^{(k)}=\Psi^k(u^{(0)})$.
For several important splittings, for instance, the classical OT split in Appendix~\ref{app-app:sinkhorn-ot} and the flow split in Section \ref{subsec-splitting-constr}, these maps have closed-form solutions, which makes this algorithm practical.

\section{Dual convergence}
\label{sec:dual-conv}

The cyclic projections of Section~\ref{subsec:cyclic} are equivalent to block--coordinate ascent on the dual objective $F_{\gamma}$ introduced in Proposition~\ref{prop:dual-gamma-correct}. This section proves the two main red-box results summarized in Figure~\ref{fig:blueprint}: Theorem~\ref{thm:kl-dual-rate}, which gives a robust $O(1/(\gamma k))$ dual convergence rate, and Theorem~\ref{thm:approx-linprog}, which converts this rate into an additive-accuracy guarantee for the unregularized linear program. The point of the blueprint is to make these robust rates modular: once primal and dual boundedness are available, the convergence proof is automatic, and Section~\ref{sec:dual-primal-dual} explains how the helper results obtain these bounds from non-expansiveness of the sweep map in the relevant quotient dual norm.
\label{subsec:mainresult}

\begin{definition}[block-quotient seminorms]\label{def:block-seminorm} For $u=(u_1,u_2)$ dual variables, we define
\[
\|u_1\|_{V_1}\eqdef\inf_{h\in\ker(A^\top)}\|u_1+h_1\|_\infty,
\qquad
\|u_2\|_{V_2}\eqdef\inf_{h\in\ker(A^\top)}\|u_2+h_2\|_\infty .
\]
The associated block-quotient dual semi-norm, which depends on the split of constraints $(A_1,A_2)$, is $\|u\|_V\eqdef\max\{\|u_1\|_{V_1},\|u_2\|_{V_2}\}$. Note that in the special case of classical OT (see Appendix~\ref{app-subsec:sinkh-analysis}), this corresponds to the so-called variation semi-norm defined in~\eqref{app-eq:variation-seminorm}.
\end{definition}

The crucial hypotheses for Theorem~\ref{thm:kl-dual-rate} are bounded primal iterates in $\ell^1$ norm and bounded dual iterates in the block-quotient semi-norm $\|\cdot\|_V$. The blueprint is designed to establish these hypotheses through two helpers: a dual bound helper based on non-expansiveness of the sweep map in the quotient norm, and a primal bound helper built on top of that dual control. This is where the split $(A_1,A_2)$ enters the analysis, since it determines the quotient geometry used to measure dual radius.

\begin{theorem}[Sub--linear dual rate]
\label{thm:kl-dual-rate}
Let $\{u^{(k)}=(u_1^{(k)},u_2^{(k)})\}_{k\ge0}$ be the dual iterates generated by the two block updates, and let $x^{(k)} \eqdef x(u^{(k)})$ be the associated primal iterates.
Assume uniform bounds $\sup_k\|u^{(k)}\|_V\le\Umax$ and $\sup_k\|x^{(k)}\|_1\le\Xmax$, where the quotient norm $\|\cdot\|_V$ is defined in Definition~\ref{def:block-seminorm}, 
and denote $\|A\|_{1\to1} \eqdef \max_{1\le j\le d}\sum_{i=1}^{m}|A_{i,j}|$.
Define the dual gap
$\Delta_k  \eqdef F_\gamma^\star-F_\gamma\!\bigl(u^{(k)}\bigr)$,  $F_\gamma^\star = \max_{u\in\RR^{m}} F_\gamma(u)$.
Then, for every $k\ge1$,
\begin{equation}\label{eq:dual-rate-new}
  0 \le \Delta_k
   \le 
  \frac{8 \Xmax{} \Umax{}^{2} \|A\|_{1\to1}^{2}}{\gamma} \frac{1}{k}.
\end{equation}
\end{theorem}

\paragraph{Proof sketch.}
The proof keeps the primal and dual viewpoints coupled. First, one half-sweep of block ascent enforces one block constraint exactly, so the remaining residual can be converted into a quantitative increase of $F_\gamma$ by strong convexity of the KL geometry on the current primal mass shell. This gives a per-step ascent bound in terms of the block residuals, with constants depending only on $\gamma$, $\Xmax$, and $\|A\|_{1\to1}$. Second, the dual gap is compared to the same residuals by testing the concave objective along a segment from the current dual variable to an optimal dual representative chosen in the quotient class; this is where the block-quotient radius $\Umax$ enters. Combining the two estimates yields a nonlinear descent recursion $\Delta_k-\Delta_{k+1}\ge \alpha\Delta_k^2$ with $\alpha=\gamma/(8\Xmax\Umax^2\|A\|_{1\to1}^2)$, and summing the reciprocal gaps gives \eqref{eq:dual-rate-new}. The complete proof is given in Appendix~\ref{app-app:dual-rate-proofs}.

\paragraph{Numerical complexity of approximating the linear program  \eqref{eq:linprog}.}
\label{subsec:cmplxity-regul}

We now turn the dual gap estimate of Theorem~\ref{thm:kl-dual-rate} into a practical stopping rule for the unregularised programme \eqref{eq:linprog}. The next theorem combines the robust dual $O(1/k)$ rate with the standard entropic approximation tradeoff, choosing the temperature $\gamma$ so that regularization and optimization errors share the target accuracy budget. The resulting guarantee uses only the primal and dual bounds $\Xmax,\Umax$.

\begin{theorem}[Accuracy versus runtime from primal/dual bounds]
\label{thm:approx-linprog}
Assume the unregularised LP \eqref{eq:linprog} admits an optimal solution $x_0^\star$ with $\|x_0^\star\|_1 \le \XmaxZero$.
Run cyclic KL projections at temperature $\gamma=\epsilon/(2 \XmaxZero\log d)$, yielding iterates $(x^{(k)},u^{(k)})$, and suppose the uniform bounds $\sup_k \|x^{(k)}\|_1 \le \Xmax$ and $\sup_k \|u^{(k)}\|_V \le \Umax$ hold (where $\|\cdot\|_V$ is defined in Definition~\ref{def:block-seminorm}). 
For $\epsilon>0$, choose
$
  k  \eqdef
  \left\lceil
  \frac{1}{\epsilon^2}
    	64 \Xmax \Umax^{2} \norm{A}_{1\to1}^{2} \max(\XmaxZero,\Xmax) \log d 
  \right\rceil , \quad
  \gamma \eqdef \frac{\epsilon}{2 \XmaxZero \log d}.
$
Then
$
  | F_0^\star - F_\gamma(u^{(k)}) |  \le  \epsilon .
$
\end{theorem}

\paragraph{Proof sketch.}
The estimate is obtained by separating approximation and optimization errors. The regularization bias compares the unregularized optimum with the entropic optimum evaluated at the same feasible point; since the KL penalty is at most $\XmaxZero\log d$ on a feasible point with mass $\XmaxZero$ after normalization against the positive reference, the choice $\gamma=\epsilon/(2\XmaxZero\log d)$ spends half of the budget on bias. The remaining half is assigned to the dual optimization gap. Inserting this value of $\gamma$ into Theorem~\ref{thm:kl-dual-rate} and requiring the displayed lower bound on $k$ makes the robust $O(1/(\gamma k))$ term at most $\epsilon/2$. The dual value, therefore, approximates the original LP value within $\epsilon$. The detailed proof is in Appendix~\ref{app-app:kl-bias-runtime}.

\paragraph{Beyond KL.}

The same blueprint extends to general Bregman divergences under a generalized Pinsker condition. Appendix~\ref{app-app:general-bregman} states the abstract condition and the resulting rate. This extension covers objectives beyond linear programs, including semidefinite programs; quantum optimal transport is one example where the non-commutative Bregman geometry is natural.


\section{Primal and Dual Bound Helpers}
\label{sec:dual-primal-dual}

Our main result, Theorem~\ref{thm:kl-dual-rate}, relies on uniform primal and dual bounds, denoted respectively by $X_{\max}$ and $U_{\max}$. We present here two helps to achieve this, shown in the blueprint of Figure~\ref{fig:blueprint} as the cyan and yellow boxes. 
 The complete proofs of both helper results are in Appendix~\ref{app-app:primaldual-proofs}.

\paragraph{Bounding Dual Iterates using Non-expansiveness of $\Psi$.}

The most difficult part is controlling the dual iterations. This section defines the geometric constants that govern this control and provides a generic blueprint that leverages the non-expansiveness of $\Psi$. It requires a primal bound $H_\gamma$ (measuring deviation from the reference $\z$ in dual coordinates) and a control of the conditioning of the splitting through a decomposition constant $\kappa$.

\begin{definition}[Primal bound $H_\gamma$ in dual coordinates]\label{def:hgamma}
We define $H_\gamma\in[0,+\infty]$ so that the minimizer $x_\gamma$ of~\eqref{eq:entropic-penalized} satisfies
$
  \|\log x_\gamma - \log \z\|_\infty \le H_\gamma.
$
\end{definition}

\begin{definition}[Decomposition constant $\kappa$]
The decomposition constant $\kappa=\kappa(A_1,A_2)$ is
\begin{equation}\label{eq:kappa}
\kappa(A_1,A_2)  \coloneqq 
\sup_{y\in\mathrm{range}(A^\top),\,y\neq 0}
\ \inf\left\{
\frac{\norminf{w_1}}{\norminf{y}}:\ \exists w_2\ \text{s.t.}\ A_1^\top w_1 + A_2^\top w_2 = y
\right\}\in[0,+\infty].
\end{equation}
\end{definition}

\begin{proposition}[Uniform $V_1$-bound for alternating maximization]\label{prop:uniform-iter-final}
We assume $\Psi$ is non-expansive with respect to $\norm{\cdot}_{V_1}$: $\norm{\Psi(a)-\Psi(b)}_{V_1}\le\norm{a-b}_{V_1}$ for all $a,b$. Let $\{u_1^{(k)}\}_{k\ge 0}$ be generated by the two block updates, and let $u_\gamma=(u_{\gamma,1},u_{\gamma,2})$ be any maximizer of $(D_\gamma)$ with associated primal optimum $x_\gamma=x(u_\gamma)$.
Then, for all $k\ge 0$,
$
	\norm{u_1^{(k)}}_{V_1}
	\le
	\norm{u_1^{(0)}}_{V_1}
	+2\kappa\left(\norm{C}_\infty+\gamma H_\gamma\right).
$
\end{proposition}

\paragraph{Proof sketch.}
The non-expansiveness assumption turns the fixed point of the sweep into an orbit center: every iterate remains within its initial distance from $u_{\gamma,1}$ in the quotient norm. It remains to bound this fixed point. At optimality, the relation $x_\gamma=z\odot\exp((A^\top u_\gamma-C)/\gamma)$ implies that $A^\top u_\gamma$ is uniformly bounded by $\|C\|_\infty+\gamma H_\gamma$ after choosing a suitable representative. The decomposition constant $\kappa$ converts this bound on $A^\top u_\gamma$ into a bound on the first block potential modulo the kernel. Combining these two steps gives the desired bound. See Appendix~\ref{app-app:primaldual-proofs} for the complete argument.

\paragraph{Bounding Primal Iterates from Dual Iterates.}
\label{sec:bounding-primal-dual}

In many cases, one directly has access to a primal bound $\Xmax$ (for instance, in classical OT, $\Xmax=1$). If this is not the case, the following proposition shows that it is always possible to derive a bound $\Xmax$ from $\Umax$, with the issue being that it blows as $\Xmax \sim \|b\|_1 \Umax/\gamma$ when  $\gamma\to 0$.
The main workload is thus to bound the dual potential in the block quotient semi-norm, which needs to be done on a case-by-case basis and exploit the structure of the split $(A_1,A_2)$.

\begin{proposition}[Primal bound from a dual bound]\label{prop:mass-bound-block}
Let $(u^{(k)})_{k\ge0}$ be the dual iterates of the two block updates and let
$x^{(k)}\coloneqq x(u^{(k)})$ be the corresponding primal iterates. Assume that the cost is lower bounded in the sense that $C_i \ge C_{\min}\ge 0$ for all $i\in\{1,\dots,d\}$, and assume for simplicity that $u^{(0)}=0$. Assume moreover a uniform block-radius bound $\Umax\coloneqq\sup_{k\ge0}\|u^{(k)}\|_V<\infty$. Then for every $k\ge0$, $\|x^{(k)}\|_1\le\Xmax\eqdef\|b\|_1\Umax/\gamma+d\,e^{-C_{\min}/\gamma}$.
\end{proposition}


\newcommand{\Geod}{D} 
\newcommand{\Flows}{\mathbb{F}} 

\section{Sinkhorn--flow algorithm for \texorpdfstring{$W_1$}{W1} on graphs}
\label{sec:applications}

This section presents a new algorithm for approximating $W_1$ on graphs and demonstrates how to instantiate the general KL-projection blueprint. In Figure~\ref{fig:blueprint}, this application is the green box, which enables the use of the primal and dual helpers to obtain robust rates. The section ends with numerical diagnostics on synthetic and genomic-inspired sparse graphs. Benchmarks are executed on CPU for reproducibility, and the same benchmark code can also be run on GPU.

\paragraph{$W_{1}$ distance on graphs.}
\label{subsec-w1-graphs}

We consider an undirected graph with a vertex set 
$\Vertex=\{1,\dots,n\}$ and edge--length matrix 
$\length\in\overline{\RR}_{+}^{\,n\times n}$ 
($\overline{\RR}_{+}\eqdef\RR_{+}\cup\{\infty\}$).  
An entry $\length_{i,j}<\infty$ indicates the presence of the edge $(i,j)$ with length $\length_{i,j}$; we impose symmetry $\length=\length^{\!\top}$. Let $\Edge = \{(i,j)\in\Vertex^{2}:\length_{i,j}<\infty\}$ and $p \eqdef |\Edge|$ be the edge set and its cardinality.  
We denote $\Geod \in\RR_{+}^{n\times n}$ the shortest--path matrix, so that $\Geod_{i,j}$ is the geodesic distance between vertices $i$ and $j$.
The graph is assumed to be connected, so that $\Geod$ is finite. 
As recalled in Appendix~\ref{app-subsec:sinkh-analysis}, for any cost matrix $C$ (for instance $C=\Geod$, the geodesic distance), the transport distance $\mathrm W_C$ is defined by the Kantorovich linear program. We consider two probability vectors $b_1,b_2$ of size $m_1=m_2=n$ and take $C=\Geod$. The resulting Optimal Transport distance $\mathrm W_{\Geod}(b_1,b_2)$ is the Wasserstein-1 ($W_1$) distance, and it enjoys an alternate linear programming formulation which leverages the sparsity of the graph adjacency matrix $W$, the so-called Beckmann formulation~\citep{Beckmann52,santambrogio2015optimal}.
We introduce \emph{flow} sparse matrices $f\in\Flows\eqdef \{ f \in \RR_{+}^{\,n\times n} \;:\; \forall (i,j) \notin \Edge, f_{i,j}=0 \}$ whose entry $f_{i,j}$ encodes the amount transported \emph{from} $j$ \emph{to}~$i$. Note that $\Flows$ has dimensionality $p$ (the set of edges). 
Define the discrete divergence operator $\diverg(f) \eqdef f^{\top}\mathbf 1_{n}-f\mathbf 1_{n}\in\RR^{n}$.
Beckmann’s theorem yields the equivalent linear program 
\begin{equation}
	\label{eq-W1-flow}
	\mathrm W_{\Geod}(b_1,b_2)
	=
	\min_{f\in\Flows} \{ \dotp{\length}{f} \;:\; \diverg(f)=b_1-b_2\}\,.
\end{equation}
Note that feasible $f$ has non--zero entries only on $\Edge$; the true variable dimension is therefore $p$ rather than $n^{2}$.  
In practice, we store flows as sparse edge lists, and all arithmetic counts in Section~\ref{subsec-splitting-constr} are expressed in terms of $p$.

\paragraph{Constraint splitting and the \textit{flow--Sinkhorn} Algorithm.}
\label{subsec-splitting-constr}

We duplicate the flow variable so that each affine block is easily projected onto. Write $x \eqdef (f,g) \in\Flows^{2}$ and $d = 2p$, where we recall $p=|\Edge|$ is the number of finite entries of $\length$ and that flows are sparse matrices with $p$ non-zero elements.  
With these lifted variables, the primal problem~\eqref{eq-W1-flow} can be re-written as 
\begin{equation}
	\label{eq-W1-flow-lifted}
	\mathrm W_{\Geod}(b_1,b_2)
	=
	\min_{(f,g) \in \Cc_{1} \cap \Cc_{2}} \dotp{\length}{f} + \dotp{\length}{g}\,. 
\end{equation}
where the constraints can be written as
$
	\Cc_{1} \eqdef \{(f,g): -f\ones_{n}+g^{\!\top}\ones_{n}=b_1-b_2\}, 
	\Cc_{2} \eqdef \{(f,g):f=g\}.
$
We then consider the entropic regularization~\eqref{eq:entropic-penalized} with a fixed reference vector $\z=(\z_{i,j})_{i,j}$ supported on the edge $(i,j) \in \Edge$
$$
		\min_{(f,g) \in \Cc_{1} \cap \Cc_{2}} \dotp{\length}{f} + \dotp{\length}{g} + \gamma \KLdiv{f}{\z}+ \gamma \KLdiv{g}{\z}.
$$

\paragraph{Primal KL projections.}

Our Sinkhorn-flow algorithm is obtained by applying the iterative KL projection to solve this regularized problem.
The KL projection on the two constraints can be computed explicitly as stated in the following proposition. 

\begin{proposition}[Closed--form KL projections]
\label{prop:graphw1-projection-closed-form}
Let $(f, g, h) \in \Flows^3$, then
\[
\Proj_{\Cc_{1}}(h,h)
=
(\diag(s)\,h,\;h\,\diag(s)^{-1}),
\qquad
\Proj_{\Cc_{2}}( f, g)
=
(\sqrt{ f\odot  g},\;\sqrt{ f\odot  g}),
\]
where the scaling is
\[
s
=
\phi\!\left(
\tfrac{b_1-b_2}{h\ones_{n}},
\tfrac{h^{\top}\!\ones_{n}}{h\ones_{n}}
\right) \in\RR_{++}^{n},
\]
with $\phi(t,u)\eqdef \frac{\sqrt{t^{2}+4u}-t}{2}$.
\end{proposition}

\paragraph{Flow-Sinkhorn algorithm.}

We call the cyclic KL scheme of Section~\ref{sec:general-KL-new} applied to the lifted variable $x=(f,g)$ the \textit{flow--Sinkhorn} algorithm. 
Since $\Proj_{\Cc_{1}}$ maps on pairs of equal flows, we track the output of this projection which simplifies the algorithmic description by focussing on a single flow variable and a single dual variable.
The flow-Sinkhorn algorithm generates flows $f^{(k)}$ converging to the solution of the entropic regularization of the initial linear program~\eqref{eq-W1-flow}.
\begin{equation}\label{eq:proj-sinkh-flow}
	(f^{(k+1)}, f^{(k+1)}) = \Proj_{\Cc_{2}} \circ \Proj_{\Cc_{1}}( f^{(k)}, f^{(k)} )
\end{equation}
We denote $v^{(k)} \in \RR^{n}$ the dual variable at iteration $k$ of the algorithm and $s^{(k)} = e^{v^{(k)}/(2\gamma)}$ the associated scaling variable, which satisfies
\begin{equation}
\label{eq:bregman-flow-form}
	f^{(k)}_{i,j} = \zC_{i,j} \, e^{\frac{v_i^{(k)}-v_j^{(k)}}{2\gamma}}, \quad
	f^{(k)} = \diag(s^{(k)}) \zC \diag(1/s^{(k)}), 
\end{equation}
where we set $\zC_{i,j} \eqdef \z_{i,j} e^{-\frac{\length_{i,j}}{\gamma}} \in \Flows$ (with the convention that $\zC_{i,j} = 0$ when there is no edge, i.e. $\length_{i,j}=+\infty$).
Proposition~\ref{prop:graphw1-flow-sinkhorn-update} gives closed forms for iterations~\eqref{eq:proj-sinkh-flow} over these dual variables.
The update of the dual is written in a stable way, so that formulas do not blow up numerically when $\gamma \to 0$, leveraging the stable implementation of the log-sum-exp operator
$
	\Ll_{\gamma}(s)\;\eqdef\;\gamma\log\!\sum_{j}\exp(s_{j}/\gamma)
$, implemented as $\Ll_{\gamma}(s-\max s)+\max s$.
This stable formula is crucial in practice when targeting small regularization.
Note that, in contrast with the usual Sinkhorn algorithm where $\zC_{i,j}>0$,  the Gibbs kernel in our setting may have vanishing entries due to the sparsity of the graph.  In this case, the classical Sinkhorn algorithm is no longer guaranteed to converge linearly,  because the Hilbert metric may fail to be contractive.  We emphasize, however, that our algorithm differs from the classical Sinkhorn method,  and our analysis does not rely on linear convergence rates.

\begin{proposition}[Flow--Sinkhorn update in scaling variables]
\label{prop:graphw1-flow-sinkhorn-update}
One has $v^{(k+1)} = \Psi( v^{(k)} )$ where $\Psi = \Psi_1 \circ \Psi_2$ (as defined in the two block updates) can be written as 
\begin{equation}
\label{eq:v-update-stable}
  \Psi(v)_i
  =
  \frac{1}{2}\,v_i
  \;+\;\frac{1}{2}\bigl(\alpha_i^+(v)-\alpha_i^-(v)\bigr)
  \;-\;\gamma \arsinh(\beta_i),
  \quad
  \alpha_i^\pm(v)\coloneqq \Ll_{\gamma}(-w_{i,\cdot} \pm v/2),
\end{equation}
where $\arsinh(m) \eqdef \log(\sqrt{1+m^2} + m)$ and 
$
	\beta_i \eqdef \frac{b_{1,i}-b_{2,i}}{2}\,e^{-\frac{\alpha_i^+(v)+\alpha_i^-(v)}{2\gamma}}$.
\end{proposition}

\paragraph{Convergence Analysis.}
\label{sec:conv-analysis-flow}

The following theorem establishes dual convergence for the graph-flow instantiation by plugging the graph-specific helper bounds into the general theorem.

\begin{theorem}[Sinkhorn--flow complexity]\label{thm:graphw1-complexity}
Sinkhorn--flow achieves an $\varepsilon$-additive approximation of the $W_1$ distance in $O\!\left(p\,\mathrm{diameter}(\Edge)^3/\varepsilon^4\right)$ operations, up to logarithmic factors in $n$, provided that $p=o(1/\log(1/\varepsilon))$.
\end{theorem}

\paragraph{Proof sketch.}
The proof instantiates the green box of Figure~\ref{fig:blueprint}. The variables are $x=(f,g)\in\Flows^2$, the two constraint blocks are the divergence constraint $A_1(f,g)=b_1-b_2$ and the equality constraint $A_2(f,g)=0$, and the dual variables are $u=(v,U)$, where $A_{1}(f,g) = f\ones_{n}-g^{\!\top}\ones_{n}\in\RR^{n}$ and $A_{2}(f,g) = f-g\in\Flows$.
Appendix~\ref{app-app:w1-proofs} identifies the quotient norms as variation seminorms and proves the signed non-expansiveness needed by the dual helper: the signature is $\Sigma=\operatorname{diag}(+I_{\Edge},-I_{\Edge})$, the translation parameter is $\tau=+1$, and the second block update satisfies $\Psi_2(v)_{i,j}=(v_j-v_i)/2$ with $\|\Psi_2(v)\|_{V_2}\le\|v\|_{V_1}$. The same appendix proves the graph decomposition estimate $\kappa\le 2\,\mathrm{diameter}(\Edge)$ and the explicit bound $|A|_{1\to1}=\|A\|_{1\to1}=2$ for this split.
The dual helper therefore gives $U_\gamma=O(\mathrm{diameter}(\Edge)(\length_{\max}+\gamma H_\gamma))$. The graph $H_\gamma$ estimate and the positive-cost lower bound then feed the primal helper, giving a mass bound $X_\gamma=O(\mathrm{diameter}(\Edge)/\gamma+p e^{-\length_{\min}/\gamma})$ for probability inputs. Plugging these $U_\gamma$, $X_\gamma$, and $|A|_{1\to1}$ values into Theorem~\ref{thm:approx-linprog}, and taking $\gamma\asymp\varepsilon$, yields $O(\mathrm{diameter}(\Edge)^3/\varepsilon^4)$ iterations up to logarithmic factors. Each sweep uses sparse edge operations, hence costs $O(p)$, which gives the stated arithmetic complexity.

\paragraph{Numerical experiments.}
The public implementation includes a PyTorch implementation of flow-Sinkhorn together with scripts reproducing these plots. We benchmark in the regime emphasized by our theory: approximation speed toward the \emph{unregularized} solution as a function of wall-clock time and regularization $\gamma$, rather than raw per-iteration decay at fixed regularization geometry. The error plotted below is the best-so-far relative Euclidean error on the recovered graph flow, $\min_{s\le t}\|f^{(s)}-f^\star\|_2/\|f^\star\|_2$, where $f^\star$ is the unregularized min-cost-flow solution and each method is evaluated through the corresponding graph-flow representation. Flow-Sinkhorn and vanilla Sinkhorn are therefore plotted together on the same axes (solid and dashed, respectively), with a common time horizon. This objective should, however, be interpreted with care for machine-learning practice: in many downstream tasks, keeping a non-vanishing (often larger) regularization can improve stability and generalization, so fastest convergence to the unregularized solution is not always the relevant end goal.
The benchmark settings are:
\begin{itemize}[leftmargin=1.2em,itemsep=1pt,topsep=2pt,parsep=0pt,partopsep=0pt]
\item \textbf{Line graph.} Segment graph with nearest-neighbor connectivity, $n=80$ nodes and two localized endpoint measures on the segment extremities; edge weights are unit lengths.
\item \textbf{Delaunay sparse graphs.} Planar point cloud with $n=140$ nodes, Delaunay triangulation connectivity, and Euclidean edge weights; localized source/target masses are sampled on spatially separated regions.
\item \textbf{Single-cell sparse graphs.} We use the Waddington-OT single-cell RNA-seq data of~\citet{schiebinger2019optimal}, which follows mouse embryonic fibroblast reprogramming toward pluripotency. The displayed graph is built on a subset of $n=240$ cells, sampled as $60$ cells from each of the first four snapshots (days $0,0.5,1,1.5$), embedded by PCA (dimension $30$) and connected by a $k$-NN graph ($k=4$). The blue and red measures are the empirical distributions at the initial and final sampled snapshots; the displayed flow is therefore a coarse transport proxy for developmental trajectories and fate progression across the sampled cell-state manifold.
\end{itemize}
For graph visualizations, we aggregate directed flows into undirected magnitudes $|f_{ij}|+|f_{ji}|$, draw all graph edges in thin gray, and overlay only upper-quantile transport edges in thicker orange; source/target supports are shown as blue/red markers. Figures~\ref{fig:bench-overview-small} and~\ref{fig:bench-line-gpu-scaling} summarize the CPU runs (for the 3 settings) and the GPU (NVIDIA Quadro RTX 4000 (8GB)) runs (for the line graph only, but with 3 different sizes $n$). In the GPU runs, increasing the line graph size makes the separation between vanilla Sinkhorn and flow-Sinkhorn more pronounced: the sparse flow formulation keeps its edge-local arithmetic, while vanilla Sinkhorn still acts on the dense path metric.

\begin{figure}[h]
\centering
\begin{minipage}{0.3\linewidth}\centering
\includegraphics[width=\linewidth]{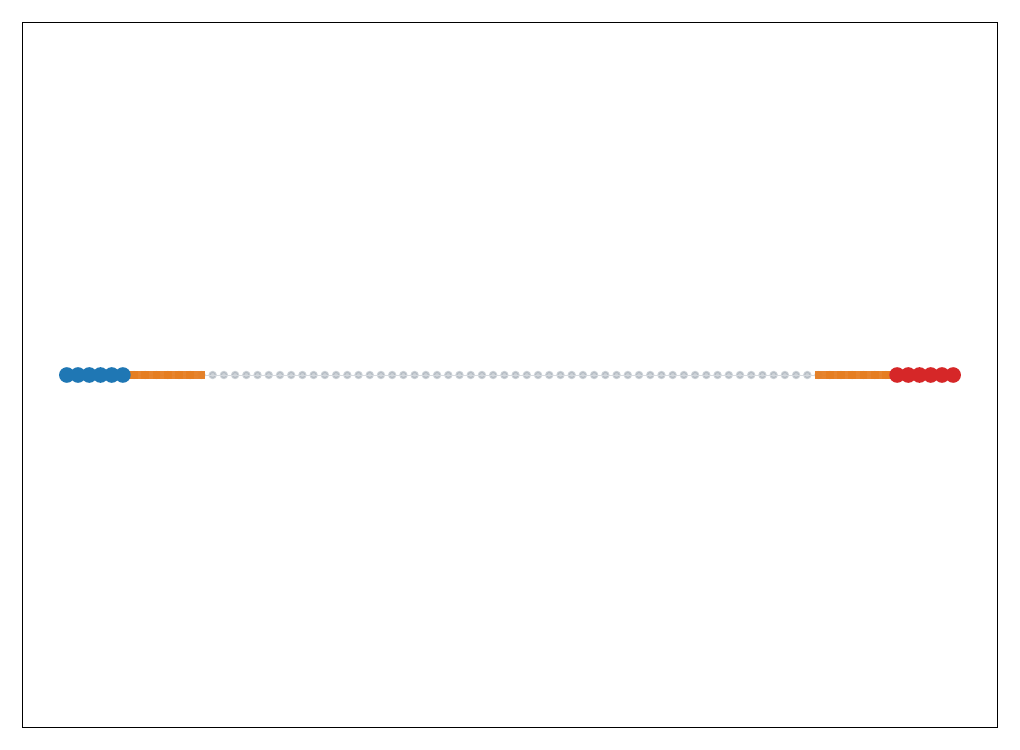}
\end{minipage}\hfill
\begin{minipage}{0.3\linewidth}\centering
\includegraphics[width=\linewidth]{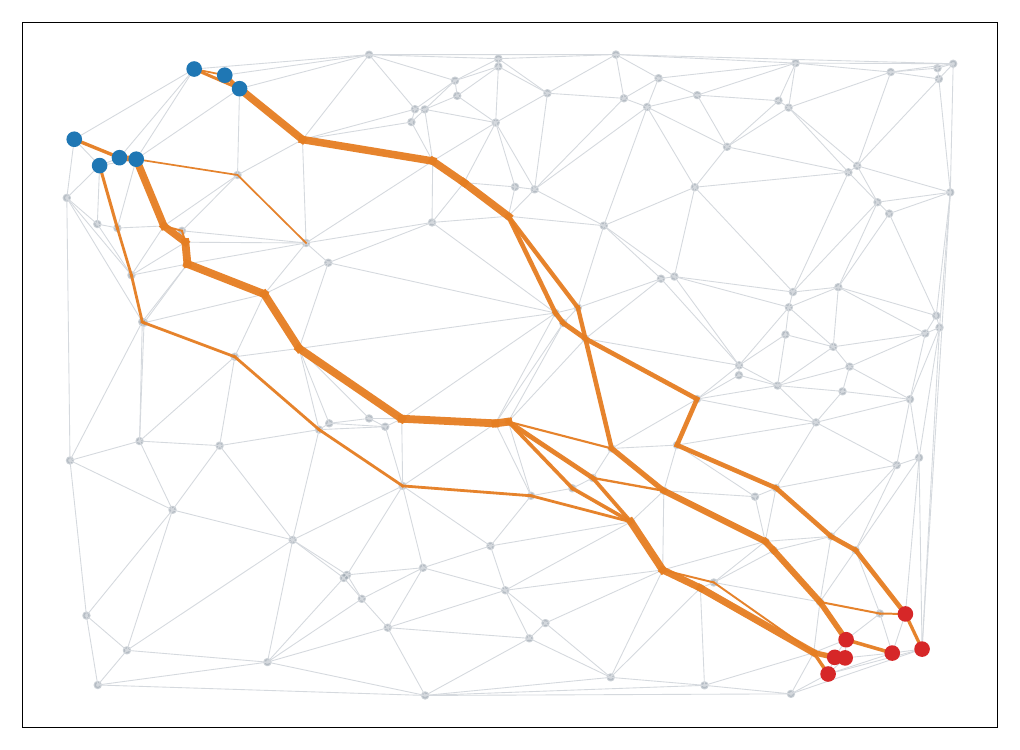}
\end{minipage}\hfill
\begin{minipage}{0.3\linewidth}\centering
\includegraphics[width=\linewidth]{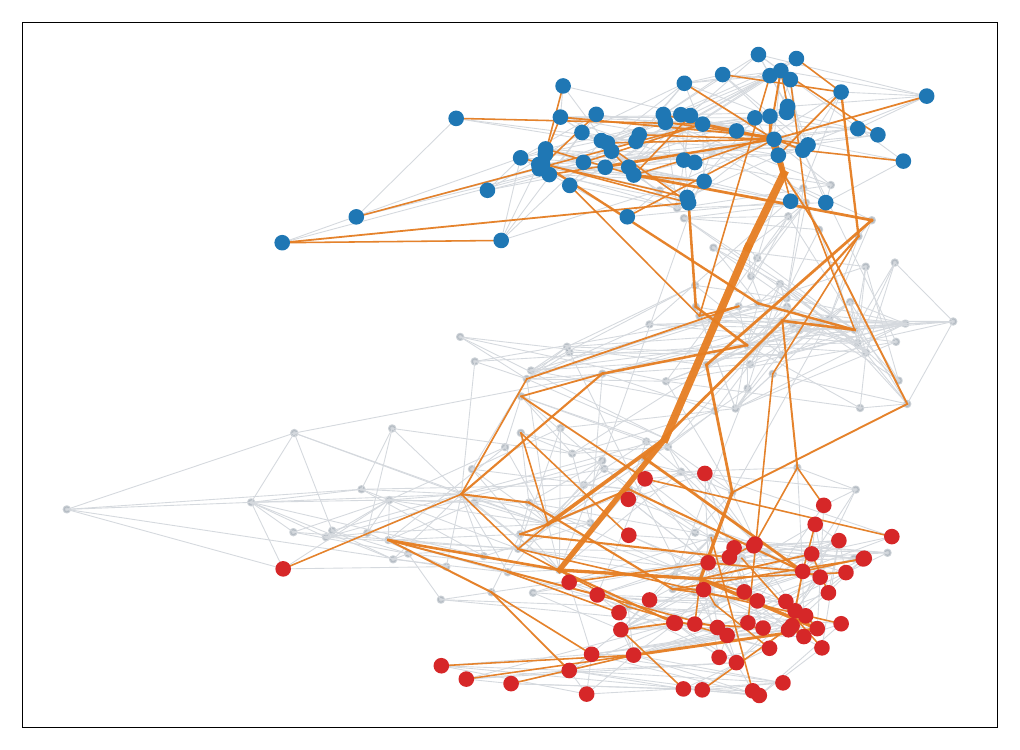}
\end{minipage}
\begin{minipage}{0.32\linewidth}\centering
\includegraphics[width=\linewidth]{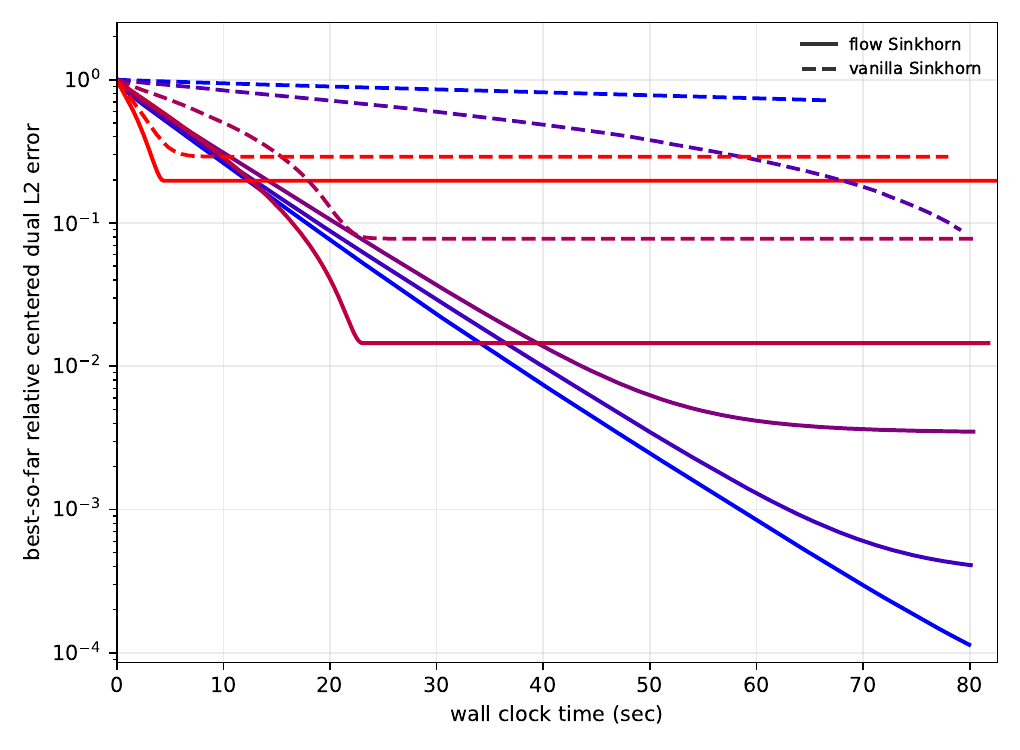}
\end{minipage}\hfill
\begin{minipage}{0.32\linewidth}\centering
\includegraphics[width=\linewidth]{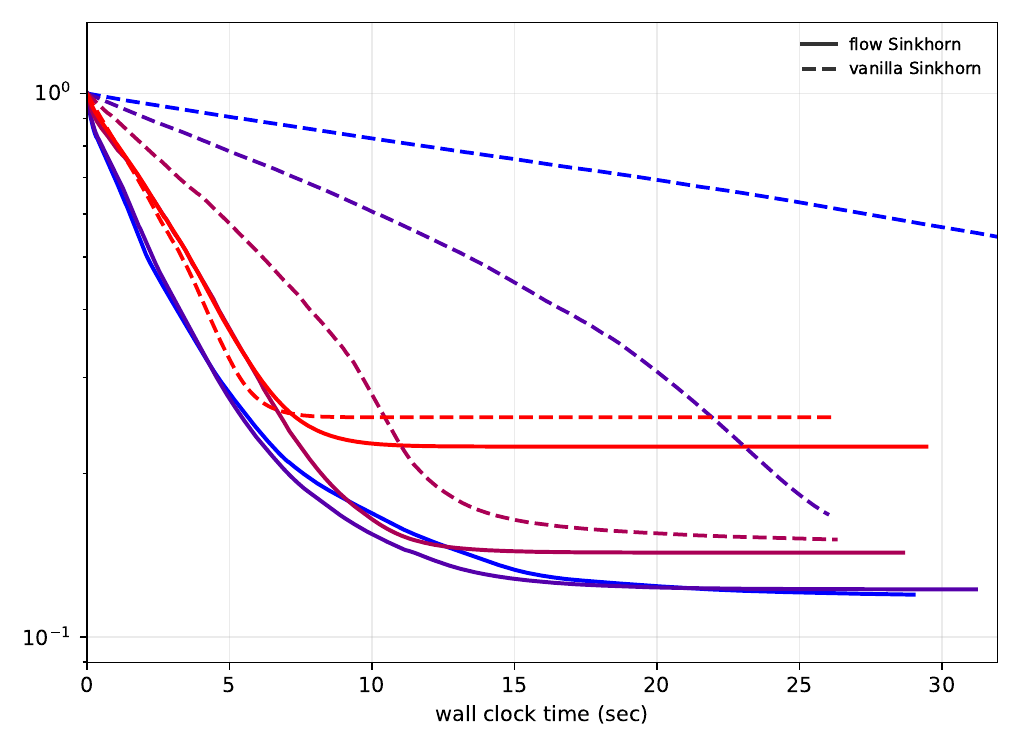}
\end{minipage}\hfill
\begin{minipage}{0.32\linewidth}\centering
\includegraphics[width=\linewidth]{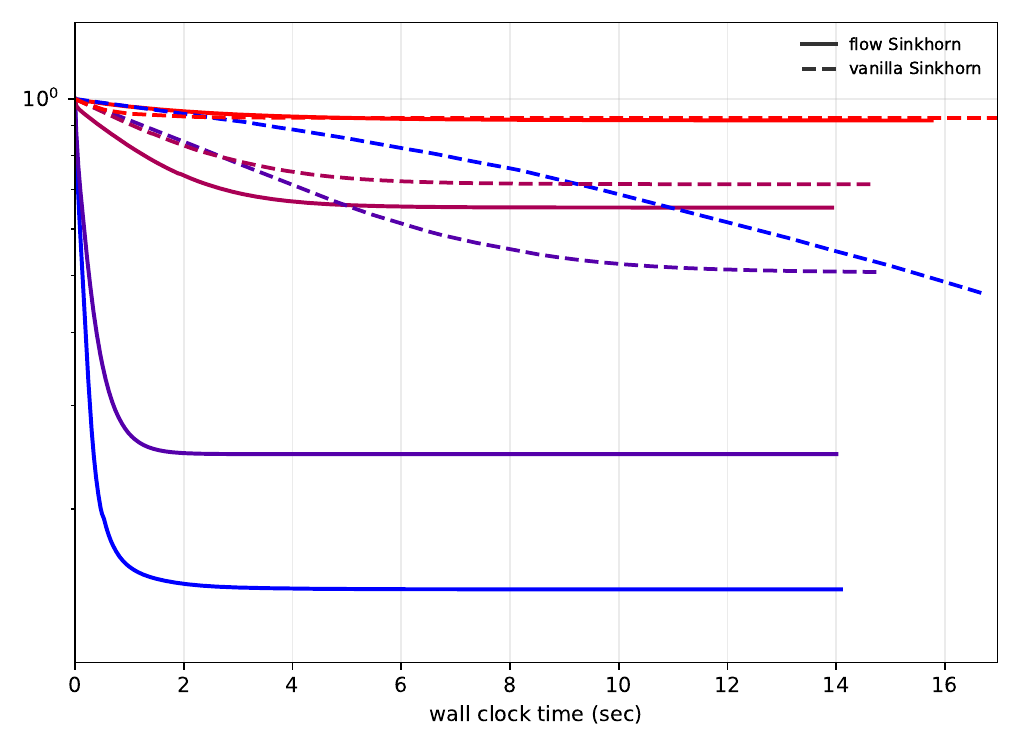}
\end{minipage}\vspace{-2mm}
\caption{Top row: benched graph together with the flow $f$ (line graph, planar Delaunay, single cell). Bottom row:$\ell^2$ error $\|f^{(s)}-f^\star\|_2/\|f^\star\|_2$, against wall-clock time. Flow Sinkhorn is plain and vanilla Sinkhorn is dashed; colors encode $\gamma$ from blue (small) to red (large). 
Gamma ranges are: line (flow $\gamma\in[10^{-3},5]$, vanilla $\gamma\in[10^{-1},1]$), Delaunay (flow $\gamma\in[9\times10^{-4},9\times10^{-3}]$, vanilla $\gamma\in[9\times10^{-4},9\times10^{-3}]$), single-cell (flow $\gamma\in[3\times10^{-1},9]$, vanilla $\gamma\in[1,5]$).}
\label{fig:bench-overview-small}
\end{figure}

\begin{figure}[h]
\centering
\begin{minipage}{0.32\linewidth}\centering
\includegraphics[width=\linewidth]{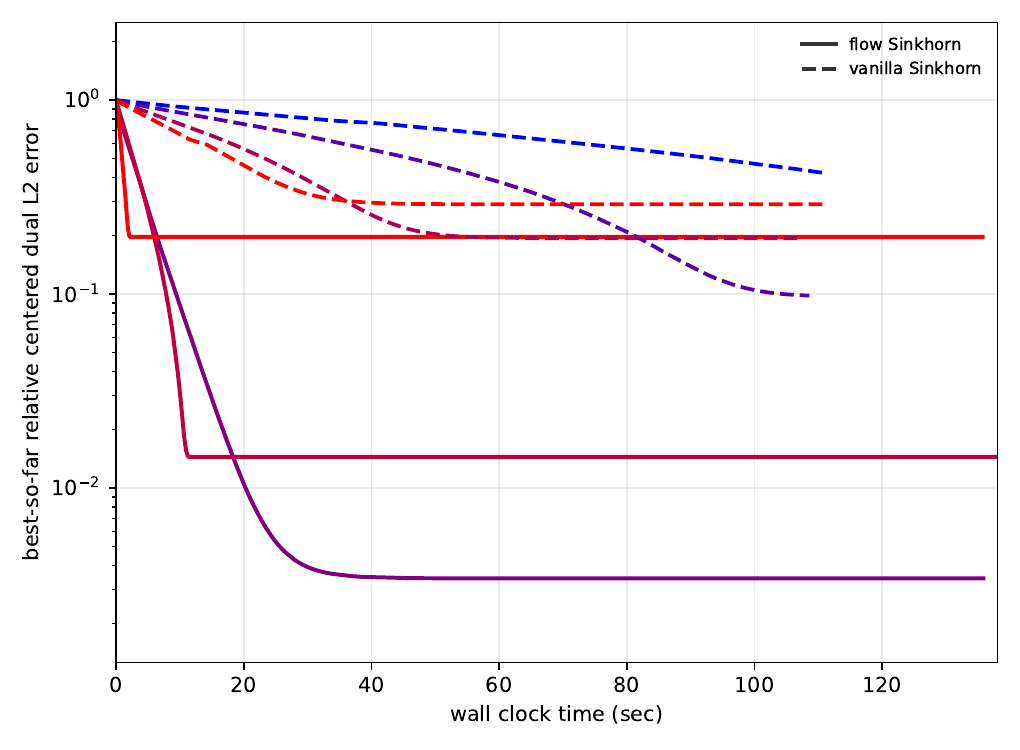}
\end{minipage}\hfill
\begin{minipage}{0.32\linewidth}\centering
\includegraphics[width=\linewidth]{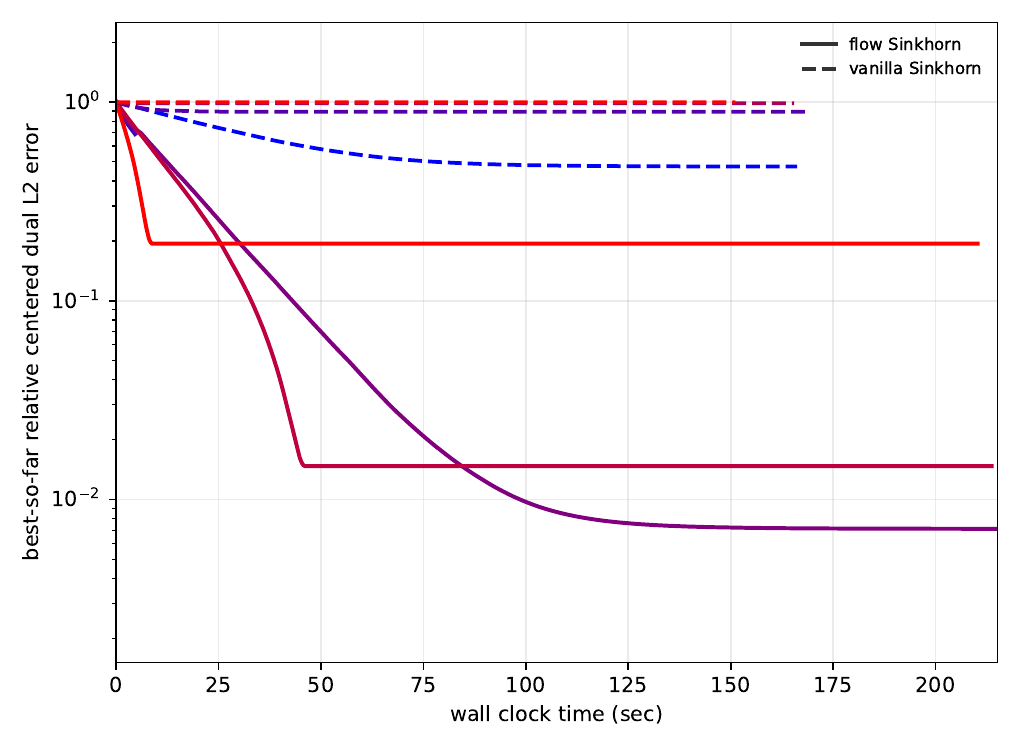}
\end{minipage}\hfill
\begin{minipage}{0.32\linewidth}\centering
\includegraphics[width=\linewidth]{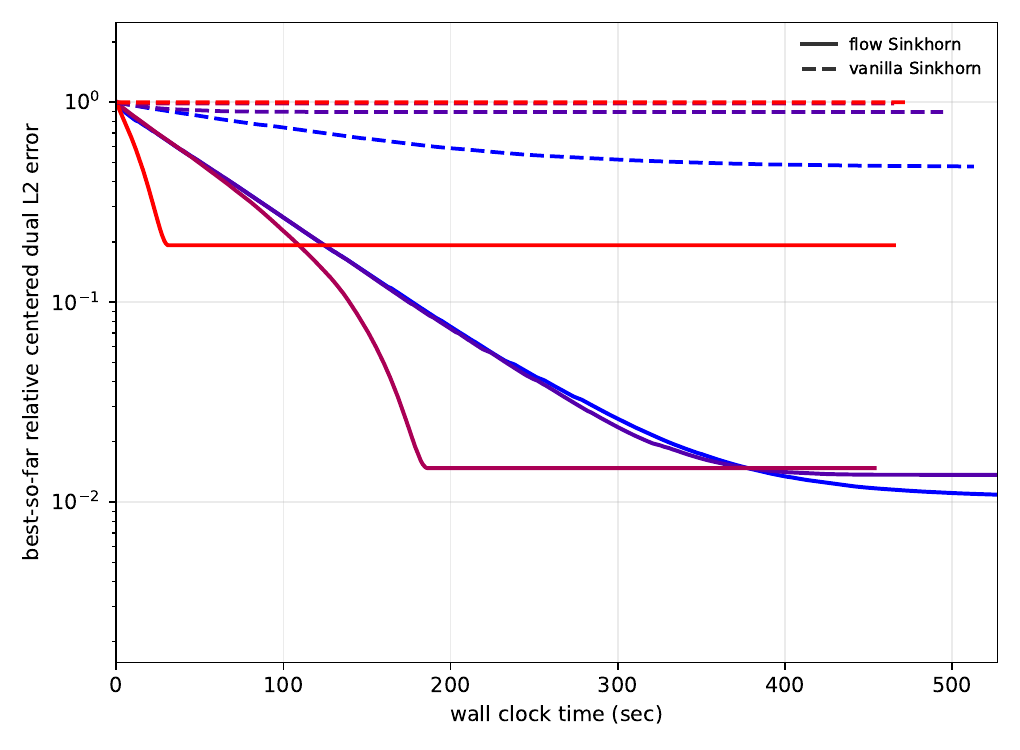}
\end{minipage}\vspace{-2mm}
\caption{GPU line-graph benchmark, same best-so-far relative $\ell^2$ flow error as in Figure~\ref{fig:bench-overview-small}. From left to right: small ($n=80,m=6$), large ($n=160,m=12$), and larger ($n=320,m=24$) line-segment tests. 
}
\label{fig:bench-line-gpu-scaling}
\end{figure}


\section*{Conclusion}

We have presented an analysis of iterative Bregman projection methods that yields sublinear convergence rates but with constants that scale favorably with the dimension. Such a regime is the relevant one for studying approximation rates in machine learning, where problems are typically obtained through sampling and discretization, and dimensional dependence plays a central role. While the behavior of these methods was well understood in the classical Optimal Transport setting, it was less clear which structural properties were truly responsible for their stability and effectiveness. This article identifies and isolates these pivotal ingredients, and proposes a generic blueprint that applies beyond the classical case. The relevance of the analysis is illustrated through its application to network flow formulations, leading to a cheap and easy-to-implement algorithm for computing the Wasserstein-1 distance on graphs.

\section*{Acknowledgement}
This work was supported by the European Research Council (ERC project WOLF) and the French government under the management of Agence Nationale de la Recherche as part of the ``France 2030'' program, reference ANR-23-IACL-0008 (PRAIRIE-PSAI).

\appendix
\section{Proof of the sublinear dual rate}
\label{app-app:dual-rate-proofs}

\paragraph{Proof of Theorem~\ref{thm:kl-dual-rate}.}
\begin{proof}
The proof relies on a per-step ascent estimate (Lemma~\ref{app-lem:per-step-ascent}) and a dual-gap--to--residual comparison (Lemma~\ref{app-lem:gap-vs-res-quotient}).
Write the global residual as $r^{(k)}=Ax^{(k)}-b$ and recall $\Xmax{} \ge \|x^{(k)}\|_1$ for all $k$.
%
%
Summing~\eqref{app-eq:A1-flow} and~\eqref{app-eq:A2-flow} in Lemma~\ref{app-lem:per-step-ascent} yields, for a full outer sweep, denoting $\lambda \eqdef \frac{\gamma}{2 \Xmax{} \|A\|_{1\to1}^2}$
\[
  \Delta_{k}-\Delta_{k+1}
   \ge 
  \lambda ( \|r^{(k)}_{1}\|_{1}^{2}+\|r^{(k+\tfrac12)}_{2}\|_{1}^{2} ) \geq \lambda \|r^{(k)}_{1}\|_{1}^{2}. 
\]
By Lemma~\ref{app-lem:gap-vs-res-quotient}, $\Delta_k \le 2\Umax{} \|r_1^{(k)}\|_{1}$, hence $\|r^{(k)}_1\|_{1}\ge \Delta_k/(2\Umax{})$. Substituting into the previous inequality gives
\[
  \Delta_{k}-\Delta_{k+1}
   \ge 
  \alpha \Delta_k^2, \quad\text{where}\quad \alpha \eqdef \frac{\lambda}{4 \Umax{}^2}.
\]
Dividing by $\Delta_k \Delta_{k+1}$ and using the fact that $\Delta_k$ is decaying gives
$  \frac{1}{\Delta_{k+1}}-\frac{1}{\Delta_{k}}
   \ge \alpha.
$
Summing from $0$ to $k-1$ yields
\[
  \frac{1}{\Delta_{k}}
   \ge 
  \frac{1}{\Delta_{0}}+\alpha k
   \ge  \alpha k,
\]
since $\Delta_{0}>0$. Therefore, for every $k\ge1$,  $\Delta_{k} \le \frac{1}{\alpha k}$,
which is exactly the estimate of Theorem~\ref{thm:kl-dual-rate}.
\end{proof}

The proof of Theorem~\ref{thm:kl-dual-rate} rests on two ingredients:  Lemma~\ref{app-lem:per-step-ascent} quantifies the ascent achieved during one half--sweep of the dual ascent the two block updates; Lemma~\ref{app-lem:gap-vs-res-quotient} relates the resulting dual gap to the primal residuals.
In the following, let $u^{(k)}$ denote the dual iterates and $x^{(k)} \eqdef x(u^{(k)})$ the corresponding primal iterates.  

\begin{lemma}[Per--step ascent for the dual blocks]
\label{app-lem:per-step-ascent}
Denoting $r^{(k)}  \coloneqq  A x^{(k)}-b  \in \RR^m$,  
assuming uniform bounds on the primal mass
$
  \sup_{k} \|x^{(k)}\|_{1} \le \Xmax{}, 
$
then for every $k\ge 0$ the block updates satisfy 
\begin{align}
  F_{\gamma}\bigl(u_{1}^{(k+1)},u_{2}^{(k)}\bigr)
  -F_{\gamma}\bigl(u_{1}^{(k)},u_{2}^{(k)}\bigr)
  & \ge 
  \frac{\gamma}{2 \Xmax{}} 
  \frac{\|r^{(k)}_{1}\|_{1}^{2}}{\|A\|_{1\to1}^{2}},
  \tag{A1}\label{app-eq:A1-flow}\\[4pt]
  F_{\gamma}\bigl(u_{1}^{(k+1)},u_{2}^{(k+1)}\bigr)
  -F_{\gamma}\bigl(u_{1}^{(k+\frac12)},u_{2}^{(k)}\bigr)
  & \ge 
  \frac{\gamma}{2 \Xmax{}} 
  \frac{\|r^{(k+1)}_{2}\|_{1}^{2}}{\|A\|_{1\to1}^{2}}.
  \tag{A2}\label{app-eq:A2-flow}
\end{align}
\end{lemma}

\begin{proof}
We prove \eqref{app-eq:A1-flow}; the argument for \eqref{app-eq:A2-flow} is identical with the roles of $(A_{1},u_{1},r_{1})$ and $(A_{2},u_{2},r_{2})$ swapped.
Let $\varphi(x)=\sum_{i} x_{i}(\log x_{i}-1)$, so $D_{\varphi}(p,q)=\KL(p\|q)$.
Because $x^{(k+\frac12)}$ is the $\KL$--projection of $x^{(k)}$ onto $\mathcal{C}_{1}=\{x:A_{1}x=b_{1}\}$, for every $s \in \mathcal{C}_{1}$,
\[
  \KLdiv{s}{x^{(k)}}
  = 
  \KLdiv{s}{ x^{(k+\frac12)} }
  +
  \KLdiv{x^{(k+\frac12)} }{ x^{(k)} }.
\]
With $x^{(k)}=x(u_{1}^{(k)},u_{2}^{(k)})$ and $x^{(k+\frac12)}=x(u_{1}^{(k+1)},u_{2}^{(k)})$, and using $A_{1}x^{(k+\frac12)}=b_{1}$ and $A_{1}x^{(k)}=b_{1}-r_{1}^{(k)}$, a direct calculation shows
\[
  F_{\gamma}\bigl(u_{1}^{(k+\frac12)},u_{2}^{(k)}\bigr)
  -F_{\gamma}\bigl(u_{1}^{(k)},u_{2}^{(k)}\bigr)
  = 
  \gamma \KL\bigl(x^{(k+\frac12)} \| x^{(k)}\bigr).
\]
By the non--normalised Pinsker inequality, stated for completeness as
Lemma~\ref{app-lem:pinsker-nonnormalized} below and applied on the common mass shell of the
KL projection,
$
  \KLdiv{p}{q} \ge  \frac{\|p-q\|_{1}^{2}}{2 \sum_{i} p_{i}}.
$
Taking $p=x^{(k+\frac12)}$ and using $\sum_{i}x^{(k+\frac12)}_{i}\le \Xmax$ gives
\[
  \KLdiv{x^{(k+\frac12)} }{ x^{(k)} }
   \ge  \frac{\|x^{(k+\frac12)}-x^{(k)}\|_{1}^{2}}{2 \Xmax}.
\]
Since $r_{1}^{(k)}=A_{1}\bigl(x^{(k)}-x^{(k+\frac12)}\bigr)$,
\[
  \|r_{1}^{(k)}\|_{1}
   \le  \|A_{1}\|_{1\to1} \|x^{(k)}-x^{(k+\frac12)}\|_{1}
   \le  \|A\|_{1 \to 1} \|x^{(k)}-x^{(k+\frac12)}\|_{1}, 
\]
where we used that $\|A_{1}\|_{1\to1} \leq \|A\|_{1 \to 1}$.
Combining the previous displays yields \eqref{app-eq:A1-flow}.
\end{proof}

If the dual iterates stay bounded, the following result relates the current dual gap to the $\ell^{1}$ residuals. 

\begin{lemma}[Dual gap versus global residual]
\label{app-lem:gap-vs-res-quotient}
Let $u^{(k)}$ be the dual iterates and $x^{(k)} \eqdef x(u^{(k)})$ the corresponding primal iterates.  
Assume the dual iterates are uniformly bounded in the block-quotient seminorm (defined in Definition~\ref{def:block-seminorm})
$  
\sup_{k} \|u^{(k)}\|_{V} \le \Umax.
$
Let $u^\star\in\arg\max F_\gamma$ and write $\Delta_k \eqdef F_\gamma(u^\star)-F_\gamma(u^{(k)})$ and $r^{(k)} \eqdef Ax^{(k)}-b$. Then for every $k\ge 0$,
\begin{equation}\label{app-eq:gap-vs-res-quotient-final2}
  \Delta_k  \le  2 \Umax \|r_1^{(k)}\|_{1}, \quad
  \Delta_{k+1/2}  \le  2 \Umax \|r_2^{(k+1/2)}\|_{1}, 
\end{equation}
\end{lemma}

\begin{proof}
Concavity and smoothness give, for any dual maximiser $u^\star$,
\[
  \Delta_k
  = F_\gamma(u^\star)-F_\gamma(u^{(k)})
   \le  \big\langle \nabla F_\gamma(u^{(k)}), u^\star-u^{(k)}\big\rangle
  =  \big\langle -r^{(k)}, u^\star-u^{(k)}\big\rangle.
\]
At the beginning of a sweep, the previous projection enforces $A_2 x^{(k)}=b_2$, hence $r_2^{(k)}=0$. Therefore
\[
  \Delta_k  \le  \big\langle -r_1^{(k)}, u_1^\star-u_1^{(k)}\big\rangle.
\]
Let $h^\star,h^{(k)}\in\ker(A^\top)$ be arbitrary. Since $r^{(k)}=[ r_1^{(k)};0 ]\in\mathrm{im}(A)$ and $\ker(A^\top)\perp\mathrm{im}(A)$, we have $\langle r^{(k)},h^\star\rangle=0$, i.e.
$\langle r_1^{(k)},h^\star_1\rangle=0$; similarly $\langle r_1^{(k)},h^{(k)}_1\rangle=0$.
Thus
\[
  \Delta_k
  =  \big\langle -r_1^{(k)}, (u_1^\star+h^\star_1)-(u_1^{(k)}+h^{(k)}_1)\big\rangle
   \le  \|r_1^{(k)}\|_1 \big(\|u_1^\star+h^\star_1\|_\infty+\|u_1^{(k)}+h^{(k)}_1\|_\infty\big),
\]
by Hölder. Taking the infimum independently over $h^\star,h^{(k)}\in\ker(A^\top)$ yields
\[
  \Delta_k  \le  \|r_1^{(k)}\|_1 \big(\|u_1^\star\|_{V_1}+\|u_1^{(k)}\|_{V_1}\big)
   \le  2 \Umax \|r_1^{(k)}\|_1,
\]
because $\|u^{(k)}\|_{V}\le\Umax$ implies $\|u_1^{(k)}\|_{V_1}\le\Umax$, and dual maximisers share the same $V$--seminorm limit so $\|u^\star\|_{V}\le\Umax$.
The second inequality is proved similarly using $\|\cdot\|_{V_2}$ in place of $\|\cdot\|_{V_1}$.
\end{proof}

\begin{proposition}[Normalized Pinsker inequality]\label{app-prop:pinsker-normalized}
For $\mu,\nu\in\Delta_d$,
\[
  \KLdiv{\mu}{\nu}
  \ge
  \frac12\|\mu-\nu\|_1^2.
\]
\end{proposition}

\begin{proof}
Using the variational representation of relative entropy,
\[
  \KLdiv{\mu}{\nu}
  =
  \sup_{f\in\RR^d}
  \left\{
    \langle \mu,f\rangle
    - \log\!\Big(\sum_i \nu_i e^{f_i}\Big)
  \right\},
\]
choose $f_i=\lambda s_i$ where $s_i=\operatorname{sign}(\mu_i-\nu_i)\in\{-1,1\}$ and $\lambda\ge0$. Then $\langle \mu-\nu,s\rangle=\|\mu-\nu\|_1$. Writing $a=\sum_{i:s_i=1}\nu_i$, $b=\sum_{i:s_i=-1}\nu_i$ gives $\sum_i \nu_i e^{\lambda s_i}=a e^\lambda+b e^{-\lambda}$ with $a+b=1$, hence
\[
  \log\!\Big(\sum_i \nu_i e^{\lambda s_i}\Big)
  \le
  \lambda\sum_i \nu_i s_i+\frac{\lambda^2}{2}.
\]
Therefore
\[
  \KLdiv{\mu}{\nu}
  \ge
  \lambda\|\mu-\nu\|_1-\frac{\lambda^2}{2}
  \qquad(\lambda\ge0),
\]
and choosing $\lambda=\|\mu-\nu\|_1$ yields the claim.
\end{proof}

\begin{lemma}[Non--normalised Pinsker inequality]\label{app-lem:pinsker-nonnormalized}
Let $p,q\in\RR_+^d$ have the same positive mass
$M=\sum_i p_i=\sum_i q_i>0$. Then
\[
  \KLdiv{p}{q}
  \ge
  \frac{\|p-q\|_1^2}{2M}.
\]
\end{lemma}

\begin{proof}
Set $\bar p=p/M$ and $\bar q=q/M$. Then $\bar p,\bar q\in\Delta_d$, so
$\KLdiv{\bar p}{\bar q}\ge \frac12\|\bar p-\bar q\|_1^2$ by Proposition~\ref{app-prop:pinsker-normalized}.
By homogeneity on a common mass shell,
$\KLdiv{p}{q}=M\,\KLdiv{\bar p}{\bar q}$ and
$\|\bar p-\bar q\|_1=M^{-1}\|p-q\|_1$.
Therefore
\[
  \KLdiv{p}{q}
  = M\,\KLdiv{\bar p}{\bar q}
  \ge
  \frac{M}{2}\|\bar p-\bar q\|_1^2
  =
  \frac{\|p-q\|_1^2}{2M},
\]
as claimed.
\end{proof}
\section{KL bias and runtime proof}
\label{app-app:kl-bias-runtime}

\begin{lemma}[KL bias]\label{app-lem:kl-bias}
Let $F_{0}^{\star}\eqdef\min\{\langle C,x\rangle:\ x\in\RR_{++}^{d},\ Ax=b\}$ be the value of \eqref{eq:linprog} and let $F_{\gamma}^{\star}$ be the value of \eqref{eq:entropic-penalized}. 
Assume there exists an optimal (unregularised) solution $x_{0}^\star$ of \eqref{eq:linprog} with
$\|x_{0}^\star\|_{1}\le \XmaxZero$.
Then
\begin{equation}\label{app-eq:bias-abs-anymass}
  0 \le F_{\gamma}^{\star}-F_{0}^{\star}
  \;\le\;
  \gamma\,\KLdiv{x_0^\star}{\z}.
\end{equation}
For simplicity, we now restrict attention to constant reference vectors $\z \propto \al \ones$, 
assuming $d\ge 3$ then
\begin{equation}\label{app-eq:bias-abs-constz}
  0 \le F_{\gamma}^{\star}-F_{0}^{\star}
  \;\le\;
  \gamma\,\XmaxZero\log d
  \quad\text{when using}\quad \z=(\XmaxZero/d)\ones.
\end{equation}
\end{lemma}

\begin{proof}
Since the KL-regularised objective is $\langle C,x\rangle+\gamma\,\KLdiv{x}{\z}$ and $\KLdiv{x}{\z}\ge 0$ for all $x\in\RR^d_{++}$, we have
\[
  F_{\gamma}^{\star}
  =
  \min_{\substack{x\in\RR^d_{++}\\ Ax=b}}
  \Bigl(\langle C,x\rangle+\gamma\,\KLdiv{x}{\z}\Bigr)
  \ge
  \min_{\substack{x\in\RR^d_{++}\\ Ax=b}}
  \langle C,x\rangle
  =F_0^\star,
\]
which gives $0\le F_\gamma^\star-F_0^\star$.
Moreover, since $x_0^\star$ is feasible for the regularised problem,
\[
  F_\gamma^\star
  \le
  \langle C,x_0^\star\rangle+\gamma\,\KLdiv{x_0^\star}{\z}
  =F_0^\star+\gamma\,\KLdiv{x_0^\star}{\z},
\]
hence $F_\gamma^\star-F_0^\star\le \gamma\,\KLdiv{x_0^\star}{\z}$, proving~\eqref{app-eq:bias-abs-anymass}.

For the final bound, specialise to $\z=\al\ones$ with $\al\eqdef \XmaxZero/d$.
Write $s\coloneqq \|x_0^\star\|_1\le \XmaxZero$. Expanding the divergence gives
\begin{align*}
  \KLdiv{x_0^\star}{\al\ones}
  &= \sum_{i=1}^d\Big(x_{0,i}^\star\log\frac{x_{0,i}^\star}{\al}-x_{0,i}^\star+\al\Big) 
  \le d\al + \Bigl(s\log\frac{s}{\al}-s\Bigr)_{+}
  \le d\al + s\log\frac{s}{\al}-s,
\end{align*}
where we use $\sum_{i=1}^d x_{0,i}^\star\log x_{0,i}^\star \le s\log s$ 
and $(t)_+\ge t$ for all $t\in\RR$, where $(t)_+ \eqdef \max\{t,0\}$.
With $\al=\XmaxZero/d$ this yields
\[
  \KLdiv{x_0^\star}{(\XmaxZero/d)\ones}
  \le \XmaxZero + (s(\log d-1))_{+}.
\]
Finally, if $d\ge 3$ then $\log d-1\ge 0$, so the positive part can be dropped and, using $s\le \XmaxZero$,
\[
  \KLdiv{x_0^\star}{(\XmaxZero/d)\ones}
  \le \XmaxZero + s(\log d-1)
  \le \XmaxZero + \XmaxZero(\log d-1)
  = \XmaxZero\log d.
\]
\end{proof}

\paragraph{Proof of Theorem~\ref{thm:approx-linprog}.}
\begin{proof}
Split the total deviation into bias and optimisation pieces:
\[
  \bigl|F_0^\star - F_\gamma(u^{(k)})\bigr|
   \le
  \underbrace{|F_0^\star - F_\gamma^\star|}_{\text{bias}}
   +
  \underbrace{|F_\gamma^\star - F_\gamma(u^{(k)})|}_{\text{optimisation}} .
\]
By Lemma~\ref{app-lem:kl-bias}, the existence of $x_0^\star$ with $\|x_0^\star\|_1\le \XmaxZero$ implies
$|F_0^\star - F_\gamma^\star| \le \gamma \XmaxZero\log d$. With $\gamma=\epsilon/(2 \XmaxZero\log d)$ this gives
$F_0^\star - F_\gamma^\star \le \epsilon/2$.
The $O(1/k)$ dual rate of Theorem~\ref{thm:kl-dual-rate} yields
$
  F_\gamma^\star - F_\gamma(u^{(k)})
   \le
  \frac{16 \Xmax \Umax^{2} \norm{A}_{1\to1}^{2}}{\gamma k}.
$
Our choice of
$k \ge 32 \Xmax \Umax^{2} \norm{A}_{1\to1}/(\gamma \epsilon)$
ensures $F_\gamma^\star - F_\gamma(u^{(k)})\le \epsilon/2$.
Summing the two contributions gives the claim.
\end{proof}
\section{Extension to general Bregman divergences}
\label{app-app:general-bregman}


The analysis presented above for the KL divergence extends naturally to a broader class of Bregman divergences. We briefly summarize the main ingredients and results.

\paragraph{Bregman divergence.}
Let $\Xx\subset \RR^d$ be a nonempty closed convex set with nonempty relative interior, and let $\phi:\mathrm{ri}(\Xx)\to \RR$ be a Legendre-type convex function (proper, lower-semicontinuous, essentially smooth and strictly convex on $\mathrm{ri}(\Xx)$).
The Bregman divergence induced by $\phi$ is defined as
\[
  D_\phi(x|y)
  \eqdef
  \phi(x)-\phi(y)-\dotp{\nabla\phi(y)}{x-y},
  \qquad x,y\in \mathrm{ri}(\Xx).
\]
When $\Xx=\RR^d_{+}$ and $\phi(x)=\sum_{i=1}^d x_i(\log x_i-1)$, we recover the (non-normalized) KL divergence.

\paragraph{Bregman-regularized linear program.}
Fix a reference point $z\in \mathrm{ri}(\Xx)$ and a temperature $\gamma>0$.
The Bregman-regularized problem reads
\begin{equation}\label{app-eq:bregman-penalized}
  \tag{$\widetilde{\mc P}_{\gamma}$}
  \min_{x\in \Cc_{1}\cap \Cc_{2}}
  \dotp{C}{x}+\gamma\,D_\phi(x|z),
\end{equation}
which can be equivalently written as a Bregman projection onto $\Cc_1\cap\Cc_2$ of the shifted reference $z_\gamma \eqdef \nabla \phi^\ast\!\bigl(\nabla\phi(z)-\tfrac{C}{\gamma}\bigr)$.

\paragraph{Dual problem.}
Define the shifted dual reference $y_\gamma \coloneqq \nabla \phi(z)-\tfrac{C}{\gamma}$.
The dual objective is
\[
  \widetilde{F}_\gamma(u)
  \;\coloneqq\;
  \langle b,u\rangle
  -
  \gamma\,\phi^\ast\!\Bigl(y_\gamma+\tfrac{A^\top u}{\gamma}\Bigr),
\]
with the primal-dual relation $x(u) = \nabla \phi^\ast\!\bigl(y_\gamma+\tfrac{A^\top u}{\gamma}\bigr)$.
The dual stationarity condition remains $A\,x(u^\star)=b$.

\paragraph{Generalized Pinsker condition.}
The key assumption replacing the Pinsker inequality is that there exists a primal confinement set $\Xx_\gamma\subset \mathrm{ri}(\Xx)$, a norm $\|\cdot\|_\ast$ on $\RR^d$, and a constant $\eta_\gamma>0$ such that
\begin{equation}\label{app-eq:C-gamma-gen}
\forall (x,y)\in \Xx_\gamma\times \Xx_\gamma,\qquad
D_\phi(x|y) \ge \eta_\gamma\,\|x-y\|_\ast^2.
\tag{$\widetilde{C}_\gamma$}
\end{equation}
For the KL divergence with $\Xx_\gamma=\{x\in\RR^d_{++}: \|x\|_1\le \Xmax\}$ and $\|\cdot\|_\ast=\|\cdot\|_1$, this holds with $\eta_\gamma=1/(2\Xmax)$.

\paragraph{Dual convergence result.}
With the generalized Pinsker condition~\eqref{app-eq:C-gamma-gen} in place, the dual convergence result (Theorem~\ref{thm:kl-dual-rate}) extends as follows.
Let $\|A\|_{\ast\to \mathcal{V}^\ast}$ denote the operator norm from $(\RR^d, \|\cdot\|_\ast)$ to the dual of the block-quotient norm.
Assume the dual iterates satisfy $\sup_k \|u^{(k)}\|_V \le \Umax$ and that the primal iterates remain in $\Xx_\gamma$.
Then
\[
  0 \le \Delta_k
  \le
  \frac{2\,\|A\|_{\ast\to \mathcal{V}^\ast}^{2}}{\gamma\,\eta_\gamma}\,
  \Umax^{2}\,
  \frac{1}{k}.
\]
In the entropic case, $\eta_\gamma = 1/(2\Xmax)$ and $\|A\|_{\ast\to \mathcal{V}^\ast} = \|A\|_{1\to1}$, recovering the bound of Theorem~\ref{thm:kl-dual-rate}.

\paragraph{Bias bound.}
For the general Bregman case, the bias between the unregularized and regularized optima satisfies
\[
  0 \le F_\gamma^\star - F_0^\star \le \gamma\,D_\phi(x_0^\star|z),
\]
where $x_0^\star$ is an optimal solution of the unregularized problem.
Combined with the dual rate above, this yields iteration complexity bounds analogous to Theorem~\ref{thm:approx-linprog}.

\begin{remark}
Beyond the entropic (KL) case, the Cressie--Read family of divergences provides another concrete example.
For $\alpha\in(0,1)$, the generator $\varphi_\alpha(s) = \frac{s^\alpha - \alpha s + (\alpha-1)}{\alpha(\alpha-1)}$ yields a separable Bregman divergence on $\RR^d_{++}$.
Condition~\eqref{app-eq:C-gamma-gen} then holds with $\eta_\gamma = \frac{1}{2d}\Xmax^{\alpha-2}$ and $\|\cdot\|_\ast = \|\cdot\|_1$.
Unlike the KL case ($\alpha=1$), the curvature-based argument introduces a dimension factor $1/d$ in the Pinsker constant.
\end{remark}

\subsection{Quantum optimal transport as a noncommutative Bregman projection}
\label{app-app:quantum-ot}

The general-Bregman framework naturally extends to entropic quantum optimal transport (QOT). This appendix presents a self-contained formulation aligned with the blueprint developed in the paper and clarifies what is already in place versus what remains open. In short, the variational and geometric setup is clear, but obtaining explicit, problem-dependent bounds for the key blueprint constants is still an open step. Matrix-valued and quantum variants of optimal transport have appeared in control and signal-processing formulations~\citep{ning2014matrix}, in the mean-field/classical limits of quantum mechanics~\citep{golse2016mean,caglioti2019quantum}, and in quantum-information and learning settings~\citep{chakrabarti2019quantum}. Entropic regularization for matrix-valued/quantum OT was developed in~\citep{peyre2019quantum}, which is the closest numerical precedent for the variational problem considered here.

\paragraph{Variational QOT problem.}
Let $\mathbb{H}_N$ denote Hermitian matrices, identify $N=nm$, and write $\operatorname{Tr}_2:\mathbb{H}_{nm}\to\mathbb{H}_n$ and $\operatorname{Tr}_1:\mathbb{H}_{nm}\to\mathbb{H}_m$ for the two partial traces. Given positive semidefinite marginals $A\in\mathbb{H}_{n,+}$ and $B\in\mathbb{H}_{m,+}$ with equal trace, and a Hermitian cost $C\in\mathbb{H}_{nm}$, the unregularized QOT problem is
\[
\mathrm W_C(A,B)
\eqdef
\min_{T\succeq 0}
\left\{
\operatorname{tr}(CT)
\;:\;
\operatorname{Tr}_2(T)=A,\ 
\operatorname{Tr}_1(T)=B
\right\}.
\]
Its dual maximizes $\operatorname{tr}(FA)+\operatorname{tr}(GB)$ under the Loewner constraint $F\otimes I_m+I_n\otimes G\le C$. The exact entropic regularization is obtained from the noncommutative entropy generator $\phi(T)=\operatorname{tr}(T\log T-T)$ and a reference $Z\succ0$:
\[
\min_{\operatorname{Tr}_2(T)=A,\ \operatorname{Tr}_1(T)=B,\ T\succ0}
\operatorname{tr}(CT)+\gamma D_{\phi}(T|Z),
\qquad
D_\phi(T|Z)
=
\operatorname{tr}\bigl(T(\log T-\log Z)-T+Z\bigr).
\]
Equivalently, with $K_\gamma=\exp(\log Z-C/\gamma)$, this is the Bregman projection of $K_\gamma$ onto the intersection of the two affine marginal constraints.

\paragraph{Dual and block maps.}
With dual variables $(F,G)\in\mathbb{H}_n\times\mathbb{H}_m$, the regularized dual objective is
\[
\mathcal F_\gamma(F,G)
=
\operatorname{tr}(FA)+\operatorname{tr}(GB)
-
\gamma\operatorname{tr}
\exp\!\left(
\frac{F\otimes I_m+I_n\otimes G-C}{\gamma}
\right)
\]
up to the harmless reference-dependent shift when $Z\ne I$. The primal-dual relation is
\[
T(F,G)
=
\exp\!\left(
\frac{F\otimes I_m+I_n\otimes G-C}{\gamma}
\right),
\]
or, for general $Z$, the same formula with $-C/\gamma$ replaced by $\log Z-C/\gamma$. Alternating maximization defines exact block maps $\Phi_1(G)=\argmax_F\mathcal F_\gamma(F,G)$ and $\Phi_2(F)=\argmax_G\mathcal F_\gamma(F,G)$. These maps are implicit in general, because matrix logarithms and exponentials do not commute; this lack of a closed form is not itself fatal for the blueprint, which only requires exact block solves and quantitative control of the resulting sweep.

\paragraph{Quantum Pinsker geometry.}
The correct analogue of the classical Pinsker inequality is the quantum Pinsker inequality, historically attributed to Hiai--Ohya--Tsukada~\citep{hiai1981sufficiency}; it also follows from Uhlmann's earlier relative-entropy interpolation results~\citep{uhlmann1977relative}, and is stated in modern textbook form in~\citep{watrous2018theory}. For trace-one $P,Q\succeq0$,
\[
D(P\|Q)
\eqdef
\operatorname{tr}\bigl(P(\log P-\log Q)\bigr)
\ge
\frac12\|P-Q\|_1^2,
\qquad
\|X\|_1\eqdef\operatorname{tr}|X|.
\]
Thus the natural primal norm in the QOT blueprint is the trace/nuclear norm, and the natural dual norm is the operator norm. For the pair of dual variables one should quotient by the gauge $(F,G)\sim(F+cI_n,G-cI_m)$, leading to the spectral-variation seminorm
\[
\|(F,G)\|_{\mathrm{var},\mathrm{op}}
\eqdef
\inf_{c\in\RR}
\max\{\|F+cI_n\|_{\mathrm{op}},\|G-cI_m\|_{\mathrm{op}}\}.
\]
This is the direct noncommutative counterpart of the variation seminorm that appears for classical Sinkhorn and graph flow-Sinkhorn.

\paragraph{Open constants for the blueprint.}
The variational dictionary is now clear: QOT fits the general-Bregman proof with the nuclear norm on primal perturbations, the quotient operator norm on dual variables, and a Pinsker constant supplied by quantum Pinsker on trace slices. What remains open is to bound the problem-dependent constants needed to apply the generic convergence theorem. In particular, one needs a usable bound on the quantum analogue of $H_\gamma$, equivalently a lower spectral bound on the optimal $T_\gamma$, and a uniform quotient-operator bound on the dual iterates.

\paragraph{Why the classical monotonicity proof does not transfer.}
In the scalar Sinkhorn setting, monotonicity plus translation equivariance of the dual maps implies non-expansiveness in the variation seminorm. In QOT the natural order is the Loewner order, but this order is not a lattice order, and the exact block maps $\Phi_1,\Phi_2$ are not operator monotone in the sense needed to reproduce the topical-map argument. Preliminary numerical and differential evidence indicates that an individual block map need not be non-expansive in the spectral-variation seminorm, although it remains open whether the full sweep could satisfy a weaker orbit-confinement or non-expansiveness estimate. New strategies are therefore needed, for instance via derivative projection-strip estimates, scalar-shifted positive decompositions, or direct strong-concavity/cocoercivity bounds in the quotient operator geometry.
\section{Proofs for primal and dual iterate bounds}
\label{app-app:primaldual-proofs}

\paragraph{Bounding Primal and Dual Iterates.}
\label{app-sec:dual-primal-dual}

Our main result, Theorem~\ref{thm:kl-dual-rate}, crucially relies on uniform primal and dual bounds, denoted respectively by $X_{\max}$ and $U_{\max}$.
We first show in Subsection~\ref{app-sec:bounding-primal-dual} that an $O(1/\gamma)$ bound on the primal variables automatically follows from a dual bound (in some cases, such as classical optimal transport, an $O(1)$ primal bound holds a priori, and this step can be skipped).
Subsection~\ref{app-sec:bounding-dual} then presents a general strategy to establish dual boundedness, based on the assumption that the sweep map $\Psi$ of the algorithm is non-expansive in the $V$-norm. This approach applies both to the classical optimal transport setting and to the optimal transport on graphs problem studied here.

\paragraph{Bounding Primal Iterates from Dual Iterates.}
\label{app-sec:bounding-primal-dual}

In many cases, one directly has access to a primal bound $\Xmax$ (for instance, in classical OT, $\Xmax=1$). If this is not the case, the following proposition shows that it is always possible to derive a bound $\Xmax$ from $\Umax$, with the issue being that it blows as $\Xmax \sim \|b\|_1 \Umax/\gamma$ when  $\gamma\to 0$.
The main workload is thus to bound the dual potential in the block quotient semi-norm, which needs to be done on a case-by-case basis and exploit the structure of the split $(A_1,A_2)$.

\paragraph{Proof of Proposition~\ref{prop:mass-bound-block}.}
\begin{proof}
Recall that $F_{\gamma}(u)=\langle b,u\rangle-\gamma\|x(u)\|_{1}$, hence
$
\|x^{(k)}\|_{1}=\frac{\langle b,u^{(k)}\rangle-F_{\gamma}(u^{(k)})}{\gamma}.
$
Because $F_{\gamma}(u^{(k)})$ is nondecreasing along the ascent, we have
$F_{\gamma}(u^{(k)})\ge F_{\gamma}(u^{(0)})=F_\gamma(0)$. Therefore,
\begin{equation}\label{app-eq:mass-first-bound}
\|x^{(k)}\|_{1}\le \frac{\langle b,u^{(k)}\rangle-F_{\gamma}(0)}{\gamma}.
\end{equation}
We first bound $\langle b,u^{(k)}\rangle$ using the block-quotient control.
Choose $h^{(k)}=(h_1^{(k)},h_2^{(k)})\in\ker(A^\top)$ such that
$\|u^{(k)}+h^{(k)}\|_{\infty}\le \|u^{(k)}\|_{V}\le \Umax$.
Since $b\in\mathrm{range}(A)$, one has $\langle b,h^{(k)}\rangle=0$, hence
\[
\langle b,u^{(k)}\rangle=\langle b,u^{(k)}+h^{(k)}\rangle
\le \|b\|_{1}\|u^{(k)}+h^{(k)}\|_{\infty}
\le \|b\|_{1}\Umax.
\]
It remains to bound $-F_\gamma(0)$ explicitly. Since $u^{(0)}=0$,
$F_\gamma(0)=-\gamma\sum_{i=1}^d \exp(\tfrac{-C_i}{\gamma} )$.
Using $C_i\ge C_{\min}$ gives $\exp(-C_i/\gamma)\le \exp(-C_{\min}/\gamma)$, hence
$F_\gamma(0)\ge -\gamma d\,e^{-C_{\min}/\gamma}$.
Plugging the two estimates into~\eqref{app-eq:mass-first-bound} yields
\[
\|x^{(k)}\|_{1}\le \frac{\|b\|_{1}\Umax+\gamma d\,e^{-C_{\min}/\gamma}}{\gamma}
=\frac{\|b\|_{1}\Umax}{\gamma}+d\,e^{-C_{\min}/\gamma},
\]
which is the mass bound stated in Proposition~\ref{prop:mass-bound-block}.
\end{proof}

\paragraph{Bounding Dual Iterates using Non-expansiveness of $\Psi$.}
\label{app-sec:bounding-dual}

The most difficult part is to control the dual iterates. In this section, we provide a generic blueprint that leverages the non-expansiveness of $\Psi$. It requires a primal bound $H_\gamma$ (measuring deviation from the reference $\z$ in dual coordinates) and a control of the conditioning of the splitting through a decomposition constant $\kappa$.

\begin{definition}[Primal bound $H_\gamma$ in dual coordinates]
We define $H_\gamma\in[0,+\infty]$ so that the minimizer $x_\gamma$ of~\eqref{eq:entropic-penalized} satisfies
\begin{equation}\label{app-eq:Hgamma}
  \|\log x_\gamma - \log \z\|_\infty \le H_\gamma.
\end{equation}
\end{definition}

\begin{definition}[Decomposition constant $\kappa$]
The decomposition constant is
\begin{equation}\label{app-eq:kappa}
\kappa=\kappa(A_1,A_2)  \coloneqq 
\sup_{y\in\mathrm{range}(A^\top),\,y\neq 0}
\ \inf\left\{
\frac{\norminf{w_1}}{\norminf{y}}:\ \exists w_2\ \text{s.t.}\ A_1^\top w_1 + A_2^\top w_2 = y
\right\}\in[0,+\infty].
\end{equation}
\end{definition}

\paragraph{Proof of Proposition~\ref{prop:uniform-iter-final}.}
\begin{proof}
Let $u_\gamma=(u_{\gamma,1},u_{\gamma,2})$ be any maximizer of $(D_\gamma)$.
Since $u_\gamma$ is globally optimal, each block is optimal given the other, hence $u_{\gamma,1}$ is a fixed point of $\Psi$:
$u_{\gamma,1}=\Psi(u_{\gamma,1})$.
Triangular inequality with $\norm{\cdot}_{V_1}$ gives
\begin{equation}\label{app-eq:split}
\norm{u_1^{(k)}}_{V_1}\le \norm{u_1^{(k)}-u_{\gamma,1}}_{V_1}+\norm{u_{\gamma,1}}_{V_1}.
\end{equation}
Using the iteration $u_1^{(k)}=\Psi^k(u_1^{(0)})$, the fixed-point property, and the non-expansiveness assumption, and triangular inequality 
\[
	\norm{u_1^{(k)}-u_{\gamma,1}}_{V_1}
	=\norm{ \Psi^k(u_1^{(0)})-\Psi^k(u_{\gamma,1}) }_{V_1}
	\le \norm{ u_1^{(0)}-u_{\gamma,1} }_{V_1}
	\le \norm{u_1^{(0)}}_{V_1}+\norm{u_{\gamma,1}}_{V_1}.
\]
Combining with~\eqref{app-eq:split} gives
\begin{equation}\label{app-eq:iter_intermediate}
	\norm{u_1^{(k)}}_{V_1}\le \norm{u_1^{(0)}}_{V_1}+2\norm{u_{\gamma,1}}_{V_1}.
\end{equation}
Let $y_\gamma \coloneqq A^\top u_\gamma\in\mathrm{range}(A^\top)$.
By the definition of $\kappa$ in \eqref{app-eq:kappa}, there exist $w_1,w_2$ such that
$A_1^\top w_1+A_2^\top w_2=y_\gamma$ and
\begin{equation}\label{app-eq:w1_kappa}
\norminf{w_1}\le \kappa\,\norminf{y_\gamma}.
\end{equation}
Set $h_1 \coloneqq w_1-u_{\gamma,1}$ and $h_2 \coloneqq w_2-u_{\gamma,2}$. Then $(h_1,h_2)\in\ker(A^\top)$, hence by definition of $\|\cdot\|_{V_1}$,
\begin{equation}\label{app-eq:u1_le_w1}
\norm{ u_{\gamma,1} }_{V_1}\le \norminf{u_{\gamma,1}+h_1}=\norminf{w_1}\le \kappa\,\norminf{y_\gamma}.
\end{equation}
It remains to bound $\norminf{y_\gamma}$.
At a dual maximizer $u_\gamma$, stationarity gives $Ax_\gamma=b$ with $x_\gamma=x(u_\gamma)$.
Moreover, by definition of $x(u)$, componentwise,
$A^\top u_\gamma = C + \gamma (\log x_\gamma - \log \z)$.
Using \eqref{app-eq:Hgamma}, we have $\|\log x_\gamma - \log \z\|_\infty \le H_\gamma$, hence
\[
\norminf{A^\top u_\gamma}
\le \norm{C}_\infty + \gamma \norminf{\log x_\gamma - \log \z}
\le \norm{C}_\infty + \gamma H_\gamma.
\]
Inserting this into~\eqref{app-eq:u1_le_w1} yields
\[
\norm{ u_{\gamma,1} }_{V_1}\le \kappa\left(\norm{C}_\infty+\gamma H_\gamma\right).
\]
Combining with~\eqref{app-eq:iter_intermediate} gives the bound stated in Proposition~\ref{prop:uniform-iter-final}.
\end{proof}
\section{Balanced optimal transport specialization}
\label{app-app:sinkhorn-ot}
\label{app-subsec:sinkh-analysis}


This section does not present new contributions, it revisits the classical entropic approximation of discrete optimal transport (OT). 
This ``warm-up'' illustrates how the constants in Theorem \ref{thm:kl-dual-rate} specialise to familiar quantities and recovers the standard $O(C_{\max}/k)$ Sinkhorn rate.  
The next subsection extends the methodology to the graph--based $W_{1}$ distance (Section~\ref{subsec-splitting-constr}).

\paragraph{Problem set-up.}

For OT, the primal vector is a transport plan $x = P \in\mathbb{R}_{++}^{m_1\times m_2}$, so that the dimension is $d=m_1 m_2$. The two marginals $b_1\in\mathbb{R}_{++}^{m_1}$ and $b_2\in\mathbb{R}_{++}^{m_2}$ are probability vectors, $\sum_i b_{1i}=\sum_j b_{2j}=1$.  
A coupling $P\in\mathbb{R}_{++}^{m_1\times m_2}$ must satisfy
\[
   A_1(P) = b_1,
   \qquad
   A_2(P) = b_2,
   \quad\text{where}\quad
   A_1(P)\coloneqq P\mathbf 1_{m_2},
   \;
   A_2(P)\coloneqq P^{\!\top}\mathbf 1_{m_1}.
\]
With a cost matrix $C\in\mathbb{R}_+^{m_1\times m_2}$ the unregularised OT instance reads
\begin{equation}\label{app-eq:ot-distance}
   \mathrm W_{C}(b_1,b_2) \eqdef 
   \min_{P \ge 0, A(P)=b}
        \langle C,P\rangle .
\end{equation}

\paragraph{Entropic dual and Sinkhorn updates.}

The primal-dual relation between solution $u$ of~\eqref{eq:dual-regul} and $P(u)$ of \eqref{eq:entropic-penalized} with reference vector $\z = b_1 \otimes b_2 =( b_{1,i} b_{2,j} )_{i,j}$ (note that we use a separable reference measure for simplicity) reads 
$$
	P(u)_{i,j} = b_{1,i} b_{2,j} e^{ \frac{u_{1,i} + u_{2,j} - C_{i,j} }{\gamma} }. 
$$
The Sinkhorn update, written over the dual variable, corresponds to the two block updates where $\Psi_1, \Psi_2$ are soft-c-transforms 
\begin{equation}\label{app-eq:soft-c-transform}
	\Psi_1(u_2)_i  \coloneqq  - \gamma \log \sum_j  e^{ \frac{-C_{i,j} + u_{2,j} }{\gamma} } b_{2,j}  ,  
	\quad 
	\Psi_2(u_1)_j  \coloneqq  - \gamma \log \sum_i   e^{ \frac{-C_{i,j} + u_{1,i}}{\gamma} } b_{1,i} .
\end{equation}
The block-quotient semi-norm is the so-called variation distance defined in~\eqref{app-eq:variation-seminorm}:
$$
	\norm{\cdot}_{V_1} = \norm{\cdot}_{V_2} = \normVar{\cdot}. 
$$
The following propositions give the value for the primal bound $H_\gamma$ of Definition~\ref{def:hgamma} and the decomposition constant $\kappa$ defined in~\eqref{eq:kappa}, which are the constants involved in the convergence rates.
Recall that for OT, the reference measure is $\z = b_1 \otimes b_2 = (b_{1,i} b_{2,j})_{i,j}$.

\begin{proposition}[$H_\gamma$ for Sinkhorn]\label{app-prop:hgamma-ot}
Assume $\min(b)>0$ and for simplicity $C_{i,j} \geq 0$. Then, for $\gamma$ small enough, one can take $H_\gamma$ in Definition~\ref{def:hgamma} as
\[
H_\gamma =
|\log(\min(b))|+\frac{2\|C\|_\infty}{\gamma}.
\]
\end{proposition}

\begin{proof}
Let $P_\gamma$ be the unique minimizer of the entropic OT problem and recall that $\z_{i,j} = (b_1)_i(b_2)_j$.
We need to bound $|\log(P_\gamma)_{i,j} - \log \z_{i,j}|$ uniformly over all $(i,j)$.

\paragraph{Upper bound.}
From the optimality conditions, $P_\gamma = \z \odot \exp((f \oplus g - C)/\gamma)$ where $f,g$ are the Sinkhorn potentials satisfying the marginal constraints. The row marginal gives
\[
\sum_j (b_2)_j \exp\left(\frac{f_i + g_j - C_{i,j}}{\gamma}\right) = 1.
\]
Since all terms are non-negative and sum to 1, each term is at most 1:
$(b_2)_j \exp\left(\frac{f_i + g_j - C_{i,j}}{\gamma}\right) \le 1$,
hence
\[
\log(P_\gamma)_{i,j} - \log \z_{i,j} = \frac{f_i + g_j - C_{i,j}}{\gamma} \le -\log(b_2)_j \le |\log(\min(b))|.
\]
The same bound follows from the column marginal using $(b_1)_i$.

\paragraph{Lower bound.}
The optimal coupling reads $P_\gamma=\diag(u) K \diag(v)$ for scaling vectors $u,v$ and $K_{i,j}  \coloneqq  e^{-C_{i,j}/\gamma}$.
Let $K_{\min}\coloneqq \min_{a,b}K_{a,b}=e^{-\|C\|_\infty/\gamma}$ (using $\min C=0$).
From $P_\gamma=\diag(u)K\diag(v)$ and the marginal constraints,
$(b_1)_i=u_i(Kv)_i$, $(b_2)_j=v_j(K^\top u)_j$,
hence
\begin{equation}\label{app-eq:factor}
(P_\gamma)_{i,j}
=u_iK_{i,j}v_j
=\frac{(b_1)_i(b_2)_j\,K_{i,j}}{(Kv)_i\,(K^\top u)_j}.
\end{equation}
Since $K\le \ones\ones^\top$ entrywise, $(Kv)_i\le \sum_\ell v_\ell$ and $(K^\top u)_j\le \sum_k u_k$.
Also, the total mass constraint gives
\[
1=\sum_{a,b}(P_\gamma)_{a,b}=u^\top Kv \ge K_{\min}\Big(\sum_k u_k\Big)\Big(\sum_\ell v_\ell\Big),
\]
so $\big(\sum_k u_k\big)\big(\sum_\ell v_\ell\big)\le 1/K_{\min}$ and therefore
$
(Kv)_i\,(K^\top u)_j \le \frac{1}{K_{\min}}.
$
Plugging this into \eqref{app-eq:factor} yields
$(P_\gamma)_{i,j}\ge (b_1)_i(b_2)_j\,K_{i,j}\,K_{\min} \ge \z_{i,j}\,K_{\min}^2$,
hence
\[
\log(P_\gamma)_{i,j} - \log \z_{i,j} \ge 2\log K_{\min} = -\frac{2\|C\|_\infty}{\gamma}.
\]
Combining both bounds gives $\|\log P_\gamma - \log \z\|_\infty \le |\log(\min(b))|+\frac{2\|C\|_\infty}{\gamma}$.
\end{proof}

\begin{proposition}[$\kappa$ for  classical OT]\label{app-prop:kappa-ot}
For classical OT, one can take $\kappa=1$ in \eqref{eq:kappa}.
\end{proposition}

\begin{proof}
Let $Y\in\mathrm{range}(A^\top)$, $Y\neq 0$. Choose any representation $Y_{ij}=\alpha_i+\beta_j$.
For any scalar $c\in\RR$, define 
$
\alpha^{(c)} \coloneqq \alpha-c\,\ones_{m_1}, \beta^{(c)} \coloneqq \beta+c\,\ones_{m_2},
$
so that $\alpha^{(c)}_i+\beta^{(c)}_j=\alpha_i+\beta_j=Y_{ij}$ for all $(i,j)$, hence
$Y=A_1^\top \alpha^{(c)}+A_2^\top \beta^{(c)}$ for every $c$.
We now show that one can choose $c$ so that $\|\alpha^{(c)}\|_\infty\le \|Y\|_\infty$.
Indeed,
\[
\min_{c\in\RR}\|\alpha-c\ones_{m_1}\|_\infty = \frac{1}{2}\left(\max_{i}\alpha_i-\min_{i}\alpha_i\right).
\]
Fix any column index $j_0\in\{1,\dots,m_2\}$. Then for all $i,i'$,
\[
\alpha_i-\alpha_{i'} = (\alpha_i+\beta_{j_0})-(\alpha_{i'}+\beta_{j_0}) = Y_{i j_0}-Y_{i' j_0}.
\]
Hence $|\alpha_i-\alpha_{i'}|\le 2\|Y\|_\infty$, which implies
$\max_i\alpha_i-\min_i\alpha_i \le 2\|Y\|_\infty$. Therefore,
$
\min_{c\in\RR}\|\alpha-c\ones_n\|_\infty \le \|Y\|_\infty.
$
Choose $c$ attaining (or arbitrarily approximating) this minimum and set $w_1 \coloneqq \alpha^{(c)}$, $w_2 \coloneqq \beta^{(c)}$.
Then $A_1^\top w_1 + A_2^\top w_2 = Y$ and $\|w_1\|_\infty\le \|Y\|_\infty$, so
\[
\inf\left\{\frac{\|w_1\|_\infty}{\|Y\|_\infty}:\ \exists w_2,\ A_1^\top w_1 + A_2^\top w_2 = Y\right\}\le 1.
\]
Taking the supremum over $Y\in\mathrm{range}(A^\top)\setminus\{0\}$ gives $\kappa\le 1$.
\end{proof}

\begin{corollary}\label{app-cor:ot-xgamma-ugamma}
For classical OT, assuming $u^{(0)}=0$ for simplicity, one may take
\[
	\Xmax = 1,
	\qquad
	\Umax =  6\,\|C\|_\infty + 2\gamma |\log(\min(b))|.
\]
\end{corollary}

\begin{proof}
We apply Proposition~\ref{app-prop:block-monotone} with $\Sigma=\mathrm{Id}$ and
Proposition~\ref{app-prop:translation-equivariance} with $\tau=-1$.
Together, these imply (via Proposition~\ref{app-prop:topical-nonexpansive})
that the sweep map $\Psi$ is non-expansive with respect to the variation seminorm
$|\cdot|_{V_1}$.
Propositions~\ref{app-prop:kappa-ot} and~\ref{app-prop:hgamma-ot} give the values of $H_\gamma$ and $\kappa$ to be used in Proposition~\ref{prop:uniform-iter-final}.
\end{proof}

This bound $\Umax$ derived using the non-expansiveness of $\Psi$ is not sharp. As shown in~\cite{chizat2020faster}, a more direct argument exploiting the closed-form expression~\eqref{app-eq:soft-c-transform} for $\Psi$, together with the Lipschitz dependence of the dual variables on the
cost matrix $C$, yields the tighter estimate $\Umax = \frac{\|C\|_\infty}{2}$.

Using Corollary~\ref{app-cor:ot-xgamma-ugamma}, Theorem~\ref{thm:approx-linprog} shows that Sinkhorn achieves
$\varepsilon$-additive accuracy on the optimal transport cost using
$	
	O( \frac{n^2}{\varepsilon^2}\,\|C\|_\infty^2 \log n )
$
arithmetic operations. This matches exactly the complexity bounds established
in~\cite{altschuler2017near,chakrabarty2021sinkhornsublinearrate,
dvurechensky2018computational,chizat2020faster}.
\section{Proofs for the graph \texorpdfstring{$W_1$}{W1} Sinkhorn-flow algorithm}
\label{app-app:w1-proofs}

\providecommand{\Geod}{D}
\providecommand{\Flows}{\mathbb{F}}

This appendix contains the proofs associated with Section~\ref{sec:applications}. We keep the notation of the main text: $\Edge$ is the directed edge set, $p=|\Edge|$, $x=(f,g)\in\Flows^2$, $\Cc_1$ is the divergence constraint, and $\Cc_2$ is the equality constraint $f=g$.

The linear operators are $A_1(f,g)=f\ones_n-g^\top\ones_n$ and $A_2(f,g)=f-g$, and for dual variables $(v,U)\in\RR^n\times\Flows$ the adjoint action is
\begin{equation}
\label{app-app-eq:At_action}
A^\top(v,U)=\Big((v_i+U_{i,j})_{(i,j)\in\Edge},\,(-v_j-U_{i,j})_{(i,j)\in\Edge}\Big)\in\Flows\times\Flows.
\end{equation}

\paragraph{Proof of Proposition~\ref{prop:graphw1-projection-closed-form}.}
For the projection onto $\Cc_1$, the KKT conditions for $\min_{f,g}\KLdiv{(f,g)}{(h,h)}$ subject to $-f\ones_n+g^\top\ones_n=b_1-b_2$ give a multiplier $\lambda\in\RR^n$ such that $\log(f_{i,j}/h_{i,j})+\lambda_i=0$ and $\log(g_{i,j}/h_{i,j})-\lambda_j=0$. Setting $s=e^{-\lambda}$ gives $f=\diag(s)h$ and $g=h\diag(s)^{-1}$. Inserting these expressions in the divergence constraint yields $(s\odot s)\odot(h\ones_n)+s\odot(b_1-b_2)-h^\top\ones_n=0$, whose positive root is the function $\phi$ of Proposition~\ref{prop:graphw1-projection-closed-form}. The projection onto $\Cc_2$ is a diagonal KL projection, hence the geometric mean.

\paragraph{Proof of Proposition~\ref{prop:graphw1-flow-sinkhorn-update}.}
Combining the two closed-form projections just proved with the parametrization $f^{(k)}=\diag(s^{(k)})z^C\diag(1/s^{(k)})$ gives the scaling update of Proposition~\ref{prop:graphw1-flow-sinkhorn-update}. Passing to $v=2\gamma\log s$ and writing the sums in log-sum-exp form gives the stable formula \eqref{eq:v-update-stable}.

\begin{proposition}[Closed forms for $\|\cdot\|_{V_1}$ and $\|\cdot\|_{V_2}$]
\label{prop:graphw1-v1v2-closed-form}
One has $\|\cdot\|_{V_1}=\normVar{\cdot}$ and $\|\cdot\|_{V_2}=\normVar{\cdot}$, where $\normVar{\cdot}$ is the variation semi-norm defined in~\eqref{app-eq:variation-seminorm}.
\end{proposition}

\begin{proposition}[Signed structure of the graph-flow split]
\label{prop:graphw1-signed-structure}
For the lifted variables $x=(f,g)$, take the diagonal signature $\Sigma=\operatorname{diag}(+I_{\Edge},-I_{\Edge})$ on the two flow blocks and the translation parameter $\tau=+1$. Then the signed constraint operators satisfy the hypotheses of the monotone-block and translation-equivariant criteria of Appendix~\ref{app-sec:non-expansiveness}. Consequently, the full sweep map $\Psi=\Psi_1\circ\Psi_2$ is non-expansive for the variation quotient norm on the $v$ block.
\end{proposition}

\begin{proof}
With the sign convention detailed below, multiplying the $g$ block by $-1$ turns each column of the two block moment maps into a nonnegative incidence contribution: the $v_i+U_{i,j}$ and $-v_j-U_{i,j}$ terms become monotone in the signed coordinates. Thus $A_s\Sigma$ is entrywise nonnegative for the two block updates. Moreover the two signed incidence contributions cancel after adding a constant to the vertex potential and the same constant to the edge multiplier, giving the paired-balance identity of Proposition~\ref{app-prop:translation-equivariance} with $\tau=+1$. Proposition~\ref{app-prop:block-monotone} gives monotonicity of each block map and Proposition~\ref{app-prop:translation-equivariance} gives translation equivariance of the sweep. Applying Proposition~\ref{app-prop:topical-nonexpansive} yields non-expansiveness in the variation quotient norm.
\end{proof}

\begin{proposition}[Closed form and non-expansiveness of $\Psi_2$]\label{prop:graphw1-psi2-closed-nonexp}
One has
\begin{equation}\label{eq:Psi2_closed_symmetric}
\Psi_2(v)_{i,j}=\frac12\,(v_j-v_i),
\quad\text{and}\quad
\|\Psi_2(v)\|_{V_2}\le \|v\|_{V_1}\qquad\text{for all }v\in\RR^n.
\end{equation}
\end{proposition}

\paragraph{Proof of Proposition~\ref{prop:graphw1-v1v2-closed-form}.}
The kernel relation $(\delta v,\delta U)\in\ker(A^\top)$ is equivalent to $\delta v_i+\delta U_{i,j}=0$ and $-\delta v_j-\delta U_{i,j}=0$ for every $(i,j)\in\Edge$. Since the graph is connected, this implies $\delta v=c\ones_n$ and $\delta U=-c\ones_\Edge$. Therefore Definition~\ref{def:block-seminorm} gives $\|v\|_{V_1}=\inf_c\|v+c\ones_n\|_\infty=\|v\|_{\mathrm{Var}}$ and $\|U\|_{V_2}=\inf_c\|U-c\ones_\Edge\|_\infty=\|U\|_{\mathrm{Var}}$.

\paragraph{Proof of Proposition~\ref{prop:graphw1-psi2-closed-nonexp}.}
Maximizing $F_\gamma(v,U)$ with respect to $U_{i,j}$ at fixed $v$ gives $\Psi_2(v)_{i,j}=(v_j-v_i)/2$. Thus $\|\Psi_2(v)\|_{V_2}=\|\Psi_2(v)\|_{\mathrm{Var}}\le \|\Psi_2(v)\|_\infty\le \frac12(\max_i v_i-\min_i v_i)=\|v\|_{V_1}$.

\begin{proposition}[$H_\gamma$ for flow Sinkhorn]\label{app-prop:hgamma-graphw1}
Assume
\[
0<\length_{\min}\coloneqq\min_{(i,j)\in\Edge}\length_{i,j},
\qquad
\length_{\max}\coloneqq\max_{(i,j)\in\Edge}\length_{i,j}.
\]
Fix any feasible $\bar f\ge0$ with $\bar f\ones-\bar f^\top\ones=b_1-b_2$.
One can take
\[
H_\gamma=\log X_\gamma^\star+\frac{2\length_{\max}}{\gamma}+3\|\log z\|_\infty,
\qquad
X_\gamma^\star\coloneqq\frac{\langle \length,\bar f\rangle+\gamma\KLdiv{\bar f}{z}}{\length_{\min}} .
\]
\end{proposition}

\begin{proof}
The positive-cost bound of Lemma~\ref{app-lem:l1-bound-from-feasible} gives $\|f_\gamma\|_\infty\le\|f_\gamma\|_1\le X_\gamma^\star$, hence $\log(f_\gamma)_{i,j}-\log z_{i,j}\le\log X_\gamma^\star+\|\log z\|_\infty$. At optimality, the opposite orientations satisfy $(f_\gamma)_{i,j}(f_\gamma)_{j,i}=z_{i,j}z_{j,i}\exp[-(\length_{i,j}+\length_{j,i})/\gamma]$. Combining this identity with the same upper bound on $(f_\gamma)_{j,i}$ yields a lower bound on $(f_\gamma)_{i,j}$, and hence the displayed value of $H_\gamma$.
\end{proof}

\begin{lemma}[Primal $\ell^1$ bound under positive costs]\label{app-lem:l1-bound-from-feasible}
Assume $C_i\ge C_{\min}>0$ for all coordinates. If $x_\gamma^\star$ solves \eqref{eq:entropic-penalized}, then for every feasible $\bar x\ge0$ with $A\bar x=b$,
\[
\|x_\gamma^\star\|_1\le X_\gamma^\star\coloneqq\frac{\langle C,\bar x\rangle+\gamma\KLdiv{\bar x}{z}}{C_{\min}}.
\]
\end{lemma}

\begin{proof}
By optimality of $x_\gamma^\star$ and nonnegativity of KL, $C_{\min}\|x_\gamma^\star\|_1\le\langle C,x_\gamma^\star\rangle\le\langle C,x_\gamma^\star\rangle+\gamma\KLdiv{x_\gamma^\star}{z}\le\langle C,\bar x\rangle+\gamma\KLdiv{\bar x}{z}$.
\end{proof}

\begin{proposition}[Graph-$W_1$ decomposition and iterate bounds]\label{app-prop:kappa-graph-diameter}\label{app-cor:graphw1-xgamma-ugamma}
Let $\mathrm{diameter}(\Edge)$ be the maximum shortest-path distance between two vertices. Then $\kappa\le2\,\mathrm{diameter}(\Edge)$. If $u^{(0)}=0$, one may take
\[
U_\gamma=4\,\mathrm{diameter}(\Edge)(\length_{\max}+\gamma H_\gamma),
\qquad
X_\gamma=\frac{\|b\|_1U_\gamma}{\gamma}+p e^{-\length_{\min}/\gamma}.
\]
\end{proposition}

\begin{proof}
For $y=A^\top(v,U)$, write $y^{(f)}_{i,j}=v_i+U_{i,j}$ and $y^{(g)}_{i,j}=-v_j-U_{i,j}$. The edge field $g_{i,j}=y^{(f)}_{i,j}+y^{(g)}_{i,j}=v_i-v_j$ satisfies $\|g\|_\infty\le2\|y\|_\infty$. Fixing a root and integrating this gradient along shortest paths gives a representative $\tilde v$ with $\|\tilde v\|_\infty\le2\,\mathrm{diameter}(\Edge)\|y\|_\infty$; defining $\tilde U_{i,j}=y^{(f)}_{i,j}-\tilde v_i$ gives $y=A^\top(\tilde v,\tilde U)$. This proves the bound on $\kappa$. Proposition~\ref{prop:graphw1-signed-structure} gives the non-expansiveness hypothesis required by Proposition~\ref{prop:uniform-iter-final}; combining it with the bound on $\kappa$ and Proposition~\ref{app-prop:hgamma-graphw1} gives the stated $U_\gamma$. Finally Proposition~\ref{prop:mass-bound-block} gives the displayed $X_\gamma$.
\end{proof}

\paragraph{Proof of Theorem~\ref{thm:graphw1-complexity}.}
Choose a spanning tree and route the signed measure $q=b_1-b_2$ along this tree: for each edge of the tree, the flow is the total signed mass of one component after cutting that edge. This produces a feasible flow $\bar f$ with total transported mass at most $\mathrm{diameter}(\Edge)\|q\|_1/2$ and cost at most $\length_{\max}\mathrm{diameter}(\Edge)\|q\|_1/2$. Therefore the unregularized optimum has an explicit feasible mass bound $X_0^\star=O(\mathrm{diameter}(\Edge))$ for probability inputs. Lemma~\ref{app-lem:l1-bound-from-feasible}, Proposition~\ref{app-prop:hgamma-graphw1}, and Proposition~\ref{app-prop:kappa-graph-diameter} give explicit $H_\gamma$, $U_\gamma=O(\mathrm{diameter}(\Edge))$, and $X_\gamma=O(\mathrm{diameter}(\Edge)/\gamma+p e^{-\length_{\min}/\gamma})$. With $\gamma\asymp\varepsilon$ and $p=o(1/\log(1/\varepsilon))$, the exponential term is lower order. Theorem~\ref{thm:approx-linprog} then requires $O(X_\gamma U_\gamma^2/\varepsilon^2)=O(\mathrm{diameter}(\Edge)^3/\varepsilon^4)$ iterations up to logarithmic factors, and each iteration costs $O(p)$ sparse edge operations.
\section{Non-expansiveness in Variation Semi-norm}
\label{app-sec:non-expansiveness}

\paragraph{Topical maps and non-expansiveness.}

We first recall a classical result of so-called ``topical maps'' in the nonlinear Perron--Frobenius/max--plus theory~\cite{candrall1980some}. This ensures the non-expansiveness of a map for the variation seminorm, which is $\ell^\infty$ norm quotiented by translation
\begin{equation}\label{app-eq:variation-seminorm}
	\normVar{v} \eqdef \inf_{c\in\RR}\|v+c\,\ones\|_\infty = 
	 \frac{1}{2} \mathrm{osc}(v)
	 \quad\text{where}\quad
	 \mathrm{osc}(v)=\max(v)-\min(v), 
\end{equation}
where $\ones$ denotes the all-ones vector of the appropriate dimension.
The remaining part of this section shows that this result can be applied to $T=\Psi$, the sweep dual mapping. This result is pivotal to show dual boundedness as exposed in Section~\ref{app-sec:bounding-dual}.

\begin{proposition}[Monotone, translation--equivariant maps are non--expansive in the $V$--seminorm]
\label{app-prop:topical-nonexpansive}
Let $T:\mathbb{R}^n\to\mathbb{R}^n$ satisfy:
\begin{enumerate}
\item \emph{Monotonicity:} $x\le y$ coordinatewise $\Rightarrow T(x)\le T(y)$ coordinatewise.
\item \emph{Translation--equivariance:} 
      $T(x+c \ones)=T(x)+c \ones$ for all $x\in\mathbb{R}^n$ and $c\in\mathbb{R}$.
\end{enumerate}
Then $T$ is non--expansive for $\normVar{\cdot}$:
\[
  \normVar{T(x)-T(y)} \le \normVar{x-y} \qquad\text{for all }x,y\in\mathbb{R}^n.
\]
\end{proposition}

\begin{proof}
Fix $x,y\in\mathbb{R}^n$ and set $d\eqdef y-x\in\mathbb{R}^n$. Let
$
  a \eqdef \min_{1\le i\le n} d_i,  b\eqdef \max_{1\le i\le n} d_i,
$
so that $a\le d_i\le b$ for every $i$, i.e. coordinatewise, 
$
  x+a \ones \le  y \le  x+b \ones
$
By monotonicity of $T$ and translation equivariance this implies
\[
  T(x+a \ones) \le  T(y) \le  T(x+b \ones)
  \quad\Rightarrow\quad
  T(x)+a \ones  \le  T(y)  \le  T(x)+b \ones,
\]
hence for each coordinate $i$,
$
  a  \le  \bigl(T(y)-T(x)\bigr)_i  \le  b .
$
Therefore
\[
  \max_i \bigl(T(y)-T(x)\bigr)_i  \le  b,
  \qquad
  \min_i \bigl(T(y)-T(x)\bigr)_i  \ge  a,
\]
and taking the oscillation gives
\[
  \mathrm{osc}\bigl(T(y)-T(x)\bigr)
   \le  b-a
  = \max_i d_i - \min_i d_i
  = \mathrm{osc}(d)
  = \mathrm{osc}(y-x).
\]
Dividing by $2$ yields the desired non--expansiveness.
\end{proof}

\paragraph{Monotonicity.}

We assume there exists a diagonal signature matrix $\Sigma=\mathrm{diag}(\sigma_1,\ldots,\sigma_d)$ with
$\sigma_i\in\{\pm1\}$ such that
\begin{equation}\label{app-eq:bipartite}
  B_s \eqdef A_s \Sigma \quad\text{is entrywise nonnegative, for } s\in\{1,2\}.
\end{equation}
Equivalently, $A_s = B_s \Sigma$ with $B_s\ge 0$ componentwise.

Using $\Sigma$ and~\eqref{app-eq:bipartite}, define a partial order $\preceq_\Sigma$ on $\RR^{m_2}$ by
\begin{equation}\label{app-eq:signed-order}
  u_2 \preceq_\Sigma v_2
  \Longleftrightarrow\
  \begin{cases}
    (B_2^\top u_2)_i \le (B_2^\top v_2)_i & \text{for all } i\text{ with } \sigma_i=+1,\\
    (B_2^\top u_2)_i \ge (B_2^\top v_2)_i & \text{for all } i\text{ with } \sigma_i=-1.
  \end{cases}
\end{equation}

\begin{proposition}[Monotonicity of block updates]\label{app-prop:block-monotone}
Assume~\eqref{app-eq:bipartite}. Then 
\begin{enumerate}[label=\textnormal{(\roman*)},itemsep=0.5em]
\item (\textit{Anti-monotonicity of $\Psi_1$}) If $u_2 \preceq_\Sigma v_2$, then
$\Psi_1(v_2)\le\Psi_1(u_2)$.
\item (\textit{Anti-monotonicity of $\Psi_2$}) If $u_1 \le v_1$, then
$\Psi_2(v_1)\preceq_\Sigma\Psi_2(u_1)$. 
\item (\textit{Monotonicity of $\Psi$}) If $u_1 \le v_1$ then $\Psi(u_1) \le \Psi(v_1)$. 
\end{enumerate}
\end{proposition}

\begin{proof}
For $u \in \RR^m$, the primal-dual relation reads 
$
  x_i(u)= \zC_i \exp( (A^\top u)_i ).
$
Introducing the moment maps $M_s$, the first-order optimality conditions read
\begin{equation}\label{app-eq:FOC}
  M_s(\Psi_s(u))=b_s
  \quad\text{where}\quad 
  M_s(u) \eqdef A_s x(u) \in \RR^{m_s}, 
\end{equation}
Using~\eqref{app-eq:bipartite}, write $A_s=B_s\Sigma$ with $B_s\ge 0$ componentwise.
Then for each $i$,
\begin{equation}\label{app-eq:x_signed}
x_i(u) = \zC_i\exp\!\Big(\frac{\sigma_i}{\gamma}\big((B_1^\top u_1)_i+(B_2^\top u_2)_i\big)\Big),
\end{equation}

\noindent\textit{(i) Anti-monotonicity of $\Psi_1$.}
Assume $u_2\preceq_\Sigma v_2$. Let
$
u_1\eqdef \Psi_1(u_2), v_1\eqdef \Psi_1(v_2).
$
Then $M_1(u_1,u_2)=b_1$ and $M_1(v_1,v_2)=b_1$ by~\eqref{app-eq:FOC}.
By Lemma~\ref{app-lem:moment-monotone}, for any fixed $w$,
\begin{equation}\label{app-eq:M1_u2_mono}
M_1(w,u_2) \le M_1(w,v_2)\qquad\text{(componentwise).}
\end{equation}
In particular, $M_1(v_1,u_2)\le M_1(v_1,v_2)=b_1$.
Suppose for contradiction that $v_1\not\le u_1$, i.e. there exists an index $j$ with $(v_1)_j>(u_1)_j$.
Since $u_1\mapsto M_1(u_1,u_2)$ is componentwise nondecreasing by Lemma~\ref{app-lem:moment-monotone},
we obtain
\[
M_1(v_1,u_2) \ge M_1(u_1,u_2)=b_1.
\]
Together with $M_1(v_1,u_2)\le b_1$, this yields a contradiction. Hence $v_1\le u_1$, i.e. $\Psi_1(v_2)\le \Psi_1(u_2)$.

\noindent\textit{(ii) Anti-monotonicity of $\Psi_2$ in the signed order.}
Assume $u_1\le v_1$. Let
\[
u_2\eqdef \Psi_2(u_1),\qquad v_2\eqdef \Psi_2(v_1).
\]
Then $M_2(u_1,u_2)=b_2$ and $M_2(v_1,v_2)=b_2$ by~\eqref{app-eq:FOC}.
By Lemma~\ref{app-lem:moment-monotone}, for any fixed $w$,
$M_2(u_1,w) \le M_2(v_1,w)$ (componentwise).
In particular, $M_2(u_1,v_2)\le M_2(v_1,v_2)=b_2$.
Now suppose for contradiction that $v_2\not\preceq_\Sigma u_2$.
Since $u_2\mapsto M_2(u_1,u_2)$ is nondecreasing in the signed order by Lemma~\ref{app-lem:moment-monotone},
the failure of $v_2\preceq_\Sigma u_2$ implies $M_2(u_1,v_2)\not\le M_2(u_1,u_2)=b_2$ (and in fact
$M_2(u_1,v_2)\ge b_2$ componentwise).
This contradicts $M_2(u_1,v_2)\le b_2$. Hence $v_2\preceq_\Sigma u_2$, i.e. $\Psi_2(v_1)\preceq_\Sigma \Psi_2(u_1)$.

\noindent\textit{(iii) Monotonicity of $\Psi$.}
Assume $u_1\le v_1$. By (ii),
$\Psi_2(v_1) \preceq_\Sigma \Psi_2(u_1)$.
Apply (i) with $u_2=\Psi_2(v_1)$ and $v_2=\Psi_2(u_1)$ to obtain
$\Psi_1(\Psi_2(u_1)) \le \Psi_1(\Psi_2(v_1))$,
i.e. $\Psi(u_1)\le \Psi(v_1)$.
\end{proof}

\begin{lemma}[Monotonicity of moment maps]\label{app-lem:moment-monotone}
Assume~\eqref{app-eq:bipartite}. Then:
\begin{enumerate}[label=\textnormal{(\arabic*)},itemsep=0.4em]
\item For any fixed $u_2$, the maps $M_1(\cdot,u_2)$ and $M_2(\cdot,u_2)$ are componentwise nondecreasing.
\item For any fixed $u_1$, the maps $M_1(u_1,\cdot)$ and $M_2(u_1,\cdot)$ are nondecreasing with respect to $\preceq_\Sigma$.
\end{enumerate}
\end{lemma}

\begin{proof}
\noindent\textit{(1) monotonicity in $u_1$.}
Fix $u_2$. If $u_1\le v_1$ componentwise, then $B_1^\top u_1\le B_1^\top v_1$ componentwise because $B_1\ge 0$.
Hence for each $i$:
\begin{itemize}[leftmargin=2em,itemsep=0.2em]
\item if $\sigma_i=+1$, then the exponent in~\eqref{app-eq:x_signed} increases, so $x_i(u_1,u_2)\le x_i(v_1,u_2)$ and
$\sigma_i x_i(\cdot) B_{s,i}$ increases (since $B_{s,i}\ge 0$);
\item if $\sigma_i=-1$, then the exponent decreases, so $x_i(u_1,u_2)\ge x_i(v_1,u_2)$, and multiplying by $\sigma_i=-1$
reverses the inequality: $\sigma_i x_i(u_1,u_2)\le \sigma_i x_i(v_1,u_2)$, hence again
$\sigma_i x_i(\cdot) B_{s,i}$ increases componentwise.
\end{itemize}
Summing over $i$ yields $M_s(u_1,u_2)\le M_s(v_1,u_2)$ componentwise for $s\in\{1,2\}$, proving (1) and (2).

\noindent\textit{(2) monotonicity in $u_2$ for the signed order.}
Fix $u_1$. If $u_2\preceq_\Sigma v_2$, then by definition~\eqref{app-eq:signed-order},
\[
(B_2^\top u_2)_i \le (B_2^\top v_2)_i \text{ when }\sigma_i=+1,
\qquad
(B_2^\top u_2)_i \ge (B_2^\top v_2)_i \text{ when }\sigma_i=-1.
\]
Equivalently,
\[
\sigma_i (B_2^\top u_2)_i \le \sigma_i (B_2^\top v_2)_i\qquad \text{for all }i.
\]
Therefore the exponent in~\eqref{app-eq:x_signed} increases for every $i$, which implies $x_i(u_1,u_2)\le x_i(u_1,v_2)$.
Now consider the contributions $\sigma_i x_i(\cdot)\,B_{s,i}$:
\begin{itemize}[leftmargin=2em,itemsep=0.2em]
\item if $\sigma_i=+1$, increasing $x_i$ increases the nonnegative vector $x_i B_{s,i}$;
\item if $\sigma_i=-1$, increasing $x_i$ decreases the nonpositive vector $-x_i B_{s,i}$, i.e.increases it componentwise.
\end{itemize}
Summing over $i$ yields $M_s(u_1,u_2)\le M_s(u_1,v_2)$ componentwise for $s\in\{1,2\}$, which is precisely
nondecreasingness with respect to $\preceq_\Sigma$.
This proves (3) and (4).
\end{proof}

\paragraph{Translation equivariance.}

Fix a sign parameter $\tau\in\{+1,-1\}$. We say that the pair $(A_1,A_2)$ satisfies the
\emph{signed paired--balance condition} (with sign $\tau$) if
\begin{equation}\label{app-eq:paired-balance-tau}
A_1^\top \ones_{m_1} + \tau\,A_2^\top \ones_{m_2} = 0  \in \RR^d.
\end{equation}
One has $\tau=-1$ for classical OT and $\tau=+1$ for the lifted $\mathrm{W}_1$ flow formulation on graphs.

\begin{proposition}[Translation equivariance under signed paired--balance]\label{app-prop:translation-equivariance}
Assume the signed paired--balance condition~\eqref{app-eq:paired-balance-tau}. Then for every $c\in\RR$,
\begin{equation}\label{app-eq:trans_Psi12}
\Psi_2(u_1+c\,\ones_{m_1})=\Psi_2(u_1)+\tau\,c\,\ones_{m_2}, \quad
\Psi_1(u_2+c\,\ones_{m_2})=\Psi_1(u_2)+\tau\,c\,\ones_{m_1}.
\end{equation}
Consequently, the full sweep $\Psi=\Psi_1\circ\Psi_2:\RR^{m_1}\to\RR^{m_1}$ is translation--equivariant:
\begin{equation}\label{app-eq:trans_Psi}
\Psi(u_1+c\,\ones_{m_1})=\Psi(u_1)+c\,\ones_{m_1}.
\end{equation}
\end{proposition}

\begin{proof}Condition~\eqref{app-eq:paired-balance-tau} implies 
\begin{equation}\label{app-eq:argmax_equivariance}
F_\gamma(u_1+c\ones_{m_1},w)=F_\gamma(u_1,w-\tau c\ones_{m_2})+\dotp{b_1}{c\ones_{m_1}}+\dotp{b_2}{\tau c\ones_{m_2}},
\end{equation}
where the last two terms are constants independent of $w$. Hence, maximizing over $w$,
\[
\arg\max_{w} F_\gamma(u_1+c\ones_{m_1},w)
=
\arg\max_{w} F_\gamma(u_1,w-\tau c\ones_{m_2}).
\]
which is the desired result for $\Psi_2$, the proof for $\Psi_1$ being similar. 
\end{proof}


\section{Lean Formalization Guide}
\label{app-sec:lean-formalization}

The Lean development is intended as an audit trail for the mathematical structure of the paper.  Each theorem, proposition, lemma, and corollary appearing in the manuscript is assigned a stable Lean alias, and the aliases are checked against the compiled Lean theorem constants.  The formalization covers the main convergence chain, the regularized approximation theorem, the OT and graph-$W_1$ instantiations, and the auxiliary ingredients used in the appendices, including KL bias bounds, quotient-seminorm estimates, non-expansiveness interfaces, and the non-normalised Pinsker reduction of Appendix~\ref{app-lem:pinsker-nonnormalized}.  These auxiliary components are deliberately stated in reusable finite-dimensional forms, so that they can be inspected independently of the particular flow-Sinkhorn application.

\paragraph{How to navigate the Lean code.}
The current umbrella entry point is \texttt{KLProjection.lean}, under
\path{lean/FlowSinkhorn/}, with project root \texttt{FlowSinkhorn.lean}.
For paper-oriented reading, \texttt{Paper.lean} imports section and appendix modules following the manuscript structure.
The canonical synchronization layer is \texttt{StatementMap.lean}, located in
\path{lean/FlowSinkhorn/KLProjection/} and re-exported from
\path{lean/FlowSinkhorn/Paper/}.
In this map, each paper-facing name, such as \texttt{thm\_3\_1}, \texttt{lem\_A\_1}, or \texttt{prop\_F\_5}, is an alias for one canonical Lean theorem constant, and each alias carries an implementation-file comment indicating where the proof is defined. The map itself is intentionally proof-free: it is a stable index from manuscript statements to proof-producing Lean modules.

The implementation modules are organized by mathematical role rather than by LaTeX order.  The main groups are duality and primal--dual identities, dual convergence and rate estimates, primal/dual uniform bounds, finite-dimensional variation geometry, and the OT and graph-$W_1$ application layers.  This organization keeps reusable proof infrastructure separate from the paper-facing alias layer while still allowing a reader to start from a statement label and jump directly to the corresponding proof file.

\paragraph{Certification scope.}
The certified development currently contains 26{,}596 non-comment, non-blank lines of Lean code in the KL-projection namespace, with 1{,}511 theorem/lemma declarations and 36 direct definition/structure declarations under the repository audit counter.  The paper-facing map is checked by scripts that verify alias completeness, compiled endpoint existence, statement numbering, and implementation-file locations.  The build is green for \texttt{cd lean \&\& lake build FlowSinkhorn.KLProjection.StatementMap}, and the KL-projection development contains no \texttt{sorry}, \texttt{admit}, or local \texttt{axiom} declarations.

A few examples illustrate what is certified.  The formalized Pinsker appendix constructs the finite sign selector, relates it to the $\ell^1$ distance, builds the two-point Bernoulli measure used for the sign test, invokes mathlib's measure-theoretic Hoeffding lemma, normalizes the common-mass variational inequality, performs the scalar quadratic optimization, and scales the result back to the non-normalised form used in Appendix~\ref{app-lem:pinsker-nonnormalized}.  The per-step ascent formalization combines this Pinsker layer with KL-gain certificates for the two block updates and then composes the two half-steps into the full sweep inequality of Lemma~\ref{app-lem:per-step-ascent}.  The quotient-residual formalization proves the finite Holder step, the shifted-representative and gauge-orthogonality manipulations, and the conversion from quotient-seminorm radii to the gap estimate used in Lemma~\ref{app-lem:gap-vs-res-quotient}.

The Lean formalization was carried out after the authors had completed and checked the LaTeX proofs. It was assisted by ChatGPT (GPT-5.4) for proof generation and proof checking. The formalization did not lead to significant changes in the authors' mathematical arguments, but it did detect and help correct arithmetic constant mistakes in intermediate bounds.

\section{Notation}
\label{app:notation}

\renewcommand{\arraystretch}{1.12}
\begin{longtable}{@{}p{0.25\textwidth}p{0.68\textwidth}@{}}
\toprule
\textbf{Notation} & \textbf{Meaning} \\
\midrule
\endfirsthead

\toprule
\textbf{Notation} & \textbf{Meaning} \\
\midrule
\endhead

$d$ & Ambient dimension of the primal variable $x\in\RR^d_+$. \\
$A=(A_1;A_2)$ & Constraint matrix split into two blocks. \\
$b=(b_1;b_2)$ & Right-hand side split compatibly with $A_1,A_2$. \\
$C\in\RR^d$ & Linear cost vector in the unregularized linear program. \\
$\z\in\RR^d_{++}$ & Positive reference measure/vector for the KL penalty. \\
$\gamma>0$ & Entropic regularization parameter. \\
$\Cc_1,\Cc_2$ & Affine constraint blocks $\{A_1x=b_1\}$ and $\{A_2x=b_2\}$. \\
$\KLdiv{x}{z}$ & Non-normalized Kullback--Leibler divergence from $x$ to $z$. \\
$\zC$ & Gibbs reference vector, $\zC_i=\z_i\exp(-C_i/\gamma)$. \\
$F_\gamma(u)$ & Dual objective associated with the entropically regularized problem. \\
$u=(u_1,u_2)$ & Dual variable split according to the two constraint blocks. \\
$x(u)$ & Primal variable recovered from a dual variable by the primal--dual relation. \\
$\Psi_1,\Psi_2$ & Exact dual block maximization maps. \\
$\Psi=\Psi_1\circ\Psi_2$ & Full dual sweep map on the first dual block. \\
$\|\cdot\|_{V_1},\|\cdot\|_{V_2}$ & Block quotient seminorms induced by $\ker(A^\top)$. \\
$\|\cdot\|_V$ & Maximum of the two block quotient seminorms. \\
$\Delta_k$ & Dual suboptimality gap $F_\gamma^\star-F_\gamma(u^{(k)})$. \\
$\Xmax$ & Uniform $\ell^1$ bound on primal iterates. \\
$\Umax$ & Uniform block-quotient bound on dual iterates. \\
$H_\gamma$ & Uniform bound on $\|\log x_\gamma-\log \z\|_\infty$ at the regularized optimum. \\
$\kappa(A_1,A_2)$ & Decomposition constant converting control of $A^\top u$ into quotient control of dual blocks. \\
$\Vertex,\Edge$ & Vertex and edge sets of the graph used for graph $W_1$. \\
$n,p$ & Number of graph vertices and directed sparse edge entries, respectively. \\
$\length$ & Edge-length/cost matrix on the graph. \\
$\Flows$ & Sparse nonnegative flow cone supported on graph edges. \\
$f,g$ & Duplicated graph-flow variables used in the flow-Sinkhorn splitting. \\
$v,U$ & Vertex and edge dual variables in the graph-flow formulation. \\
$\operatorname{diameter}(\Edge)$ & Maximum shortest-path distance between graph vertices. \\
$\Sigma,\tau$ & Signature matrix and scalar balance parameter used for signed monotonicity and translation equivariance. \\
$D_\phi$ & General Bregman divergence generated by a convex function $\phi$. \\
$\eta_\gamma$ & Generalized Pinsker constant in the Bregman extension. \\

\bottomrule
\end{longtable}

\bibliographystyle{plainnat}
\bibliography{references}
\end{document}